\documentclass[11pt]{article}
\usepackage{geometry}                
\usepackage{graphicx}
\usepackage{amsmath,amssymb,amsthm,amsrefs}
\usepackage{epstopdf}
\usepackage{mdframed}
\usepackage{nicematrix}
\usepackage{mathtools}

\usepackage{setspace}
\newtheorem{theorem}{Theorem}[section]
\newtheorem{corollary}[theorem]{Corollary}
\newtheorem{lemma}[theorem]{Lemma}
\theoremstyle{plain}
  \newtheorem{thm}{Theorem}[section]
  
  \newtheorem{prop}[theorem]{Proposition}
  
\theoremstyle{definition}
  \newtheorem{defn}[thm]{Definition}
  \newtheorem{rmk}[thm]{Remark}
  \newtheorem{ex}[thm]{Example}
\newtheorem{definition}[theorem]{Definition}
\usepackage{tikz-cd}
\DeclareMathOperator{\im}{im}
\DeclareMathOperator{\coker}{coker}
\DeclareMathOperator{\rk}{rank}
\DeclareMathOperator{\tC}{Cone}
\DeclareMathOperator{\supp}{supp}
\def\om{\omega}
\def\bko{b_k^\omega}
\def\mP{{\cal{P}}}
\def\com{c(\omega)}
\def\del{\partial}
\def\w{\wedge}
\def\Om{\Omega}
\def\CM{\mathcal{M}}

\def\tf{\tfrac{1}{2}}
\def\tsigma{\tilde{\sigma}}
\def\txw{r}
\def\tx{\tilde{x}}
\def\tom{\tilde{\om_t}}
\def\td{{\tilde{d}}}
\numberwithin{equation}{section}

\begin{document}

\title{
\bf\
{Symplectic Morse Theory and Witten Deformation}
}

\author{David Clausen, Xiang Tang and Li-Sheng Tseng\\
\\
}

\date{August 14, 2025}

\maketitle

\begin{abstract}
On symplectic manifolds, we introduce a Morse-type complex with elements generated by pairs of critical points of a Morse function.  The differential of the complex consists of gradient flows and an integration of the symplectic structure over spaces of gradient flow lines.  Using the Witten deformation method, we prove that the cohomology of this complex is independent of both the Riemannian metric and the Morse function used to define the complex and is in fact isomorphic to the cohomology of differential forms of Tsai, Tseng and Yau (TTY).  We also obtain Morse-type inequalities that bound the dimensions of the TTY cohomologies by the number of Morse critical points and the interaction of symplectic structure with the critical points.
    
\end{abstract}

\tableofcontents


\section{Introduction}

The Morse complex, also referred to as the Morse-Witten or Smale-Thom complex, captures the information of the standard homology groups of a closed manifold $M$ by means of a Morse function $f$, i.e. a function whose Hessian at each critical point is non-degenerate, and a Riemannian metric $g$.  The elements of the complex $C^k(M,f)$ are generated by the critical points $q\in Crit(f)$ of the Morse function $f$, and grouped together by their index, $k =ind(q)$, the number of negative eigenvalues of the Hessian matrix of $f$ at $q$.  The differential of the complex requires the use of the metric and is given by the gradient flow, $-\nabla f$, from one critical point to another.  We will assume throughout the paper that the gradient flow satisfy the Morse-Smale transversality condition, that is, the stable and unstable manifolds of any two critical points intersect transversely.  The homology of the Morse complex is well-known to be isomorphic to the standard homology, and therefore, independent of the choice of the Morse function $f$ and the metric $g$.  As a corollary of this isomorphism, the Morse inequalities bound the Betti numbers of $M$ in terms of the number of critical points of the Morse function.

We are interested here to consider Morse theory in the presence of a symplectic structure, that is, on a symplectic manifold $(M^{2n}, \om)$.  On the cohomology side, besides the de Rham cohomology, Tsai, Tseng and Yau (TTY) \cite{TY1, TY2, TTY}  found novel symplectic cohomologies of differential forms.   These cohomologies, which we will call TTY cohomologies and labelled by $F^pH(M, \om)$, with $p=0, 1, \ldots, n-1$, have interesting properties.  For one, they can in general vary with the symplectic structure as seen in explicit examples of a six-dimensional nilmanifold \cite{TY2} and of a three-ball product with a three-torus, $B^3 \times T^3$ \cite{TW}. These cohomologies have also been used to distinguish inequivalent symplectic structures on open 4-manifolds \cite{GTV}.  Of particular relevance here, when the symplectic structure is integral class $[\om]\in H^2(M, \mathbb{Z})$, Tanaka and Tseng pointed out that the TTY cohomologies are isomorphic to the de Rham cohomologies of a higher dimensional sphere bundle over the symplectic manifold \cite{TT}.  
Specifically, denote by $E_p$ the odd-dimensional sphere bundle  $S^{2p+1}\rightarrow  E_p\rightarrow M$ with Euler class $\om^{p+1}$, then $F^pH(M, \om) \cong H_{dR}(E_p)\,$.  Certainly, on this sphere bundle, which is a smooth manifold, we can bound the dimensions of the de Rham cohomology $H_{dR}(E_p)$ and hence, $F^pH(M, \om)$, by Morse or Morse-Bott inequalities.

To simplify the discussion, we will focus mostly in this paper on the $p=0$ TTY cohomology, $PH(M)$, called {\it primitive} cohomology, and introduced in \cite{TY2}.  (The case of $p\geq 1$ can be straightforwardly generalized from the $p=0$ case and will be described explicitly in the concluding section of this paper.)  By \cite{TT}, when $\om$ is an integral class, $PH(M)$ are isomorphic to the de Rham cohomology of the prequantum circle bundle $X$, i.e. a circle bundle with Euler class given by $\om$.  A bound on the dimension $PH(M)\cong H_{dR}(X)$ can be obtained by taking a Morse function $f$ on $M$ and pulling it back to the circle bundle, which makes $\pi^* f$ a Morse-Bott function.  Denote by $\bko = \dim PH^k(M)$.  The Morse-Bott inequalities for a circle bundle states the existence of a polynomial $Q(s)$ with positive coefficients such that
\begin{align*}
(1+s)\sum_{k=0}^{2n} m_k\, s^k = \sum_{k=0}^{2n+1} \bko\, s^k \, +\, (1+s) Q(s) 
\end{align*}
where $m_k$ denotes the number of critical points with index equal to $k$.  Specifically, this gives the strong inequalities
\begin{align}\label{MBrel1}
\sum_{i=0}^k (-1)^{k-i}b^\omega_i \leq \sum_{i=0}^k (-1)^{k-i}(m_i+m_{i-1})=m_k
\end{align}
and the weak inequalities
\begin{align}\label{MBrel2}
\bko \leq m_k + m_{k-1}\,.
\end{align}
Though these Morse-Bott inequalities \eqref{MBrel1}-\eqref{MBrel2} assume $\om$ is an integral class, they in fact hold true for any symplectic structure.  We recall the algebraic relation for the primitive cohomologies in \cite{TTY}. 
\begin{align}\label{omckk}
PH^k(M) \cong \coker\left[\om:H_{dR}^{k-2}(M) \to H_{dR}^{k}(M)\right] \oplus \ker\left[\om: H_{dR}^{k-1}(M) \to H_{dR}^{k+1}(M)\right]\,
\end{align}
which immediately gives the weak inequalities of \eqref{MBrel2} just by bounding the dimensions of $H_{dR}^k(M)$ and $H_{dR}^{k-1}(M)$ by the number of Morse critical points, $m_k$ and $m_{k-1}$, respectively.  The strong inequalities can be similarly attained by applying the rank-nullity theorem.  Hence, we find that the Morse-Bott inequalities only provide a rough estimate for $\bko$.  Moreover, note that the $\bko$'s on the left-hand-side of \eqref{MBrel1}-\eqref{MBrel2} generally depend on $\om$, while the $m_k$'s on the right hand side do not.  
These observations make us ask two questions:  
\begin{itemize}
\item[(1)] Can we find a Morse-type complex that incorporate the symplectic structure $\om$ explicitly and whose cohomology matches that of the TTY cohomology?
\item[(2)]Can we bound $\bko=\dim PH^k(M,\om)$ by inequalities that in general vary with $\om$?
\end{itemize}

\medskip

In this paper, we answer both questions in the affirmative.  For the first, we introduce a symplectic Morse complex on $(M^{2n}, \om)$  defined by a Morse-Smale pair $(f, g)$  on $M$ whose cohomology are isomorphic to the TTY primitive cohomology.  Our symplectic Morse complex is motivated by the result of Tanaka-Tseng \cite{TT} which relates the cochain complex that underlies the TTY cohomologies with the cone complex of the wedge product map $\om^{p+1}: \Om^{\bullet}(M)\to \Om^{\bullet+2p+2}(M)\,$ on the space of differential forms.  Let us recall the definition of a cone on the de Rham complex with respect to $\om^{p+1}$.  Again, for simplicity, we will focus on the case of $p=0$.  

\begin{defn} \label{Cdef}
Let $(M^{2n},\om)$ be a symplectic manifold.  We define the {\bf de Rham cone complex} of $\om$, $\tC(\om)= (\Om^\bullet(M) \oplus \theta\, \Om^{\bullet-1}(M),\, d_C)$:
\begin{equation*}
\begin{tikzcd}
\ldots \arrow[r, "d_C"]  &
\Om^{k}(M)\oplus \theta\,\Om^{k-1}(M) \arrow[r,"d_C"] & \Om^{k+1}(M)\oplus \theta\,\Om^{k}(M) \arrow[r,"d_C"]  & \ldots
\end{tikzcd}
\end{equation*}
where $\theta$ is a formal parameter of degree one and the differential $d_C$ can be expressed in matrix form as 
\begin{align}\label{dCdef}
d_C=\begin{pmatrix} d & \om \\ 0 & - d \end{pmatrix}\,
\end{align}
with $d$ the standard exterior derivative and $\om$ acting by wedge product.  
\end{defn}

\medskip

Note that the $d$-closedness of $\om$ together with the Leibniz product rule ensures that $d_C\, d_C = 0$.  Also, if we formally define $d\theta = \om\,$, then $d_C$ is just the exterior derivative acting on $\Om^\bullet(M) \oplus \theta\, \Om^{\bullet-1}(M)$.  Of interest, Tanaka-Tseng \cite{TT} proved the isomorphism of this cone cohomology with the TTY cohomology: 
\begin{align}\label{CFIso}
H(\tC(\om)) \cong PH(M, \om)\,.
\end{align}

Motivated by the relationship between de Rham complex and the Morse cochain complex over $\mathbb{R}$, we define in the following a {\it cone} Morse complex with respect to $\om$ also over $\mathbb{R}$.  
\begin{defn}\label{Ccdef}
Let $(M^{2n}, \om)$ be a symplectic manifold  equipped with a Riemannian metric $g$ and a Morse function $f$ satisfying the Morse-Smale transversality condition.  Let $C^k(M,f)$ be the $\mathbb{R}$-module with generators the critical points of $f$ with index $k$.  We define the {\bf cone Morse cochain complex} of $\om$, $\tC(c(\om))=(C^\bullet(M,f)\oplus C^{\bullet-1}(M,f),\, \del_C)$:
\begin{equation*}
\begin{tikzcd}
\ldots \arrow[r, "\del_C"]  &
C^{k}(M,f)\oplus C^{k-1}(M,f) \arrow[r,"\del_C"] & C^{k+1}(M,f)\oplus C^{k}(M,f) \arrow[r,"\del_C"]  & \ldots
\end{tikzcd}
\end{equation*}
with 
\begin{align}\label{delCdef}
\del_C= \begin{pmatrix} \del & c(\om) \\ 0 & -\del \end{pmatrix}\,.
\end{align}
Here, $\partial$ is the standard Morse cochain differential defined by gradient flow, and $c(\om): C^k(M,f)\to C^{k+2}(M,f)$ acting on a critical point of index $k$ is defined to be
\begin{align}\label{cpsidef}
c(\om)\,q_{k}=\sum_{ind(r)={k+2}}\left( \int_{\overline{\CM(r_{k+2}, q_k)}}\, \om \right)r_{k+2}
\end{align}
where $\CM(r_{k+2}, q_k)$ is the two-dimensional subspace of $M$ consisting of all flow lines from the index $k+2$ critical point, $r_{k+2}\,$, to $q_k$.
\end{defn}

\medskip

Notice that the elements of the Morse cone complex, $\tC^k(\com)=C^k(M,f)\oplus C^{k-1}(M,f)$, can be generated by {\it pairs} of critical points of index $k$ and $k-1$.  The differential $\del_C$ consists of the standard Morse differential $\partial$ from gradient flow coupled with the $c(\om)$ map which involves an integration of $\om$ over the space of gradient flow lines.  This type of maps has appeared in Austin-Braam \cite{AB} and Viterbo \cite{Viterbo} to define a cup product on Morse cohomology.  It satisfies the following commuting relation:
\begin{align}\label{MLeibniz}
\del\left( c(\om)\right)=c(\om)\del\,.
\end{align}
(This relation follows from a more general Leibniz type formula.  See \cite[Appendix A]{CTT3}.)  With \eqref{MLeibniz} and $\del \del=0\,,$ they together imply $\del_C\,\del_C=0$.  The cone Morse complex hence gives the following cohomology for $k=0, 1, \ldots, 2n+1$,
\begin{align*}
H^k(\tC(\com)) = \dfrac{\ker \del_C \cap \tC^k(\com)}{\im \del_C \cap \tC^k(\com)}\,.
\end{align*}
For this cone Morse cohomology, we are able to prove the following theorem.

\smallskip

\begin{thm}\label{MIso}
Let $(M^{2n},\om)$ be a closed symplectic manifold, The cohomology of symplectic Morse complex is isomorphic to the TTY cohomology, i.e. $H(\tC(c(\om)))\cong PH(M)$.
\end{thm}

\smallskip

Theorem \ref{MIso} importantly shows that the cohomology of the cone Morse complex is independent of the choice of both the Morse function and the Riemannian metric used to define $\tC(c(\om))$. 
Moreover, the dependence on the symplectic structure is explicit in the differential $\del_C$ which involves the  integration of $\om$ over flow lines between pairs of critical points of the Morse function.  

We shall prove Theorem \ref{MIso} by means of the analytic method of Witten deformation, as developed in \cite{BL, BZ, Zhang}, to the $\tC(\om)$ complex.  Analogous to the Witten deformation of the de Rham complex \cite{Witten}, we deform the cone differential $d_C$ and its adjoint $d_C^*$ by the Morse function $f$ parameterized by a real parameter $t\geq 0$: 
\begin{align*}
d_{C,t} = e^{-tf}d_C\,e^{tf}\,,\qquad  d^*_{C,t}= e^{tf} d^*_C\,e^{-tf}\,.   
\end{align*}
This deformation is just a conjugation by $e^{tf}$, and hence, the cohomologies of the deformed cone complex $(\tC(\om), d_{C,t})$ do not vary with $t$.  However, the deformation has a significant effect on the harmonic forms, which provides an isomorphic description of the cohomologies.  Specifically, the deformed cone Laplacian operator depends on $t$  
\begin{align*}\Delta_{C,t}=d_{C,t}d_{C,t}^*+d_{C,t}^*d_{C,t}\,, 
\end{align*}
and in fact, its highest order $t$ dependence is given by $t^2||df||_x^2$.  Hence, as $t\rightarrow \infty$, the harmonic forms must localize near $df=0\,$, i.e. the critical points of $f$. This allows an identification, as $t\rightarrow \infty$, of each harmonic form of  $\Delta_{C,t}$ with a critical point  $p\in Crit(f)$.  However, for finite $t$ sufficiently large, some of the harmonic form in the limit of $t\rightarrow \infty$
becomes no longer harmonic.  It is thus useful to consider the cone subspace $F_{C,t}^{[0,1]}\subset \tC(\om)$ consisting of eigenforms of $\Delta_{C,t}$ with eigenvalue $\lambda \in [0,1]$.  Note that $(F_{C,t}^{[0,1]}, d_{C,t})$ is also a cochain complex 
since, $[\Delta_{C,t}, d_{C,t}]=0\,$.  Of interest, for $t$ sufficiently large, the number of generators of $F_{C,t}^{[0,1]}$ matches exactly that of $\tC(\com)\,$.  
This leads us to establish an isomorphism between the cohomologies 
 $H(\tC(\com))\cong 
H(F_{C,t}^{[0,1]})\cong 
 H(\tC(\om))\cong PH(M)$, proving Theorem \ref{MIso}.  
 This successfully addresses the first question.

For the second question, we use the established quasi-isomorphism to obtain Morse-type inequalities for $H(\tC(\om))\cong PH(M)$.  With $m_k$ denoting the number of Morse critical points with index $k$, we are able to prove the following:
\begin{thm}\label{CMineq}
Let $(M, \om, f, g)$ be a closed, symplectic manifold with the Morse function $f$ and the Riemannian metric $g$ satisfying the Morse-Smale transversality condition.   Then
there exists a polynomial $Q(s)$ with non-negative integer coefficients such that
\begin{align}\label{pIneq}
(1+s)\sum_{k=0}^{2n} m_k\, s^k - (s+s^2)\sum_{k=0}^{2n-2} v_k\, s^k= \sum_{k=0}^{2n+1} b^{\om}_k\, s^k \, +\, (1+s)\, Q(s) \,,
\end{align}
where $b^\om_k=\dim H^k(\tC(\om))$ and $v_k = \rk \left(\com: C^k(M, f) \to C^{k+2}(M,f)\right)$.

Alternatively, we have the following Morse-type inequalities: \\
(A) Weak cone Morse inequalities 
\begin{align}
b_k^{\om} \ \leq \,
m_k-v_{k-2}+m_{k-1}-v_{k-1}\,, \quad\quad  k = 0, \dots, 2n+1\,;
\end{align}
(B) Strong cone Morse inequalities
\begin{align}
\sum_{i=0}^k (-1)^{k-i}b^\om_i\, \leq 
\, m_k - v_{k-1}\,,\qquad\qquad  k = 0, \dots, 2n+1\,.
\end{align}
Furthermore, the above inequalities become equalities when the Morse function $f$ is perfect, i.e. the Betti numbers $b_k=\dim H^k_{dR}(M)=m_k$ for all $k=0, \ldots, 2n\,$.
\end{thm}

\smallskip

As was our goal for the second question, our cone Morse inequalities \eqref{wcMI}-\eqref{scMI} can certainly vary with the symplectic structure $\om$.  Specifically, the $b^{\om}_k\,$'s on the left-hand side and the $v_k$'s on the right-hand side both are defined with dependence on $\om\,$.  This is in contrast to the $m_k$'s which are fixed by the choice of the Morse function $f$ on $M$. 

Beyond addressing our two main questions, let us point out that our analytic study of the symplectic cone Morse theory should be extendable to analyze symplectic manifolds with group actions.  In particular, it would be interesting to study manifolds with hamiltonian group action and work out its equivariant or more general invariant cone cohomologies and also their corresponding cone Morse theory.  
We note of a recent work \cite{HaoZhuang} that used Witten deformation to study torus actions and their related invariant cohomologies. 
Moreover, the Witten deformation method has been successfully applied (see, for example \cite{MaZ, TiZ})  to prove Guillemin-Sternberg's conjecture \cite{GS} concerning the commutativity of symplectic reduction and geometric quantization.  It is interesting to ask how symplectic reduction affect the TTY cohomologies $F^pH(M, \om)$, or more directly its cone equivalent $H(\tC(\om^{p+1}))$, and consider a different type of quantization of $(M, \omega)$ making use of the symplectic cone complex.  This paper represents a first step in addressing these other interesting questions.  

Finally, we mention that the de Rham cone complex and its Morse theory defined here with respect to the symplectic structure $\om$ has a generalization that can be studied in a very general context. In a companion paper \cite{CTT3}, we describe a general cone complex and its cone Morse theory on any oriented manifold $M$, with respect to any degree $\ell$ form $\psi\in \Om^\ell(M)\,$ that is $d$-closed.  It is a challenge to carry out the analytic Witten deformation method in this general setting, and hence, the discussion in \cite{CTT3} utilizes purely algebraic methods.

\

\noindent{\it Acknowledgements.~} 
We thank Hiro Lee Tanaka,  Weiping Zhang, and Jiawei Zhou for helpful discussions.
The second author was supported in part by NSF Grants DMS-1800666 and DMS-1952551.  The third author would like to acknowledge the support of the Simons Collaboration Grant No. 636284.

\section{Witten deformation method}
We apply the Witten deformation method (see for example, Zhang's book \cite{Zhang}) to analyze the cone complex $\tC(\omega)$.  In this section, we will introduce the deformed cone Laplacian and analyze its harmonic solutions and give a bound for the eigenvalues of non-harmonic eigenforms.  We shall begin first with some preliminaries and also introduce our notations.  

\subsection{Cone Laplacian and its deformation}

Let $(M^{2n}, \om)$ be a closed symplectic manifold and let $g$ be a compatible Riemannian metric.  For $\eta_k, \eta'_k\in \Om^k(M)$, we have the standard inner product,
\begin{align}\label{Minprod}
\langle\eta_k, \eta'_k\rangle = \int_M \eta_k \w * \eta'_k \,,
\end{align}
where $*:\Om^k(M)\rightarrow \Om^{2n-k}(M)$ is the Hodge star operator. As introduced in the Introduction, we are interested in the  cone forms with respect to the $\om\w$ map:
\begin{align*}
\tC^k(\om)(M)=\Omega^k(M)\oplus \theta\, \Omega^{k-1}(M)=\left\{\eta_k+\theta\xi_{k-1} \Big|\, \eta_k \in \Omega^k(M), \xi_{k-1} \in \Omega^{k-1}(M)\right\}, \end{align*}
with $k=0, 1, 2, \ldots, 2n+1\,$, and $\theta$ should be thought of as a formal one-form parameter, with the following two  properties: (1) $d\theta=\om$ and (2) $\theta\w\theta = 0\,$.  The cone forms are essentially a pair of differential forms, and so the standard inner product \eqref{Minprod} on $M$ can be used to define a natural inner product on $\tC^*(\om)(M)$, 
\begin{align}\label{cinprod}
\langle\eta_k + \theta \xi_{k-1}, \eta'_k + \theta \xi'_{k-1}\rangle_C \,= \, \langle\eta_k, \eta'_k\rangle + \langle\xi_{k-1}, \xi'_{k-1}\rangle\,. 
\end{align}
This cone inner product can also be expressed in terms of a Hodge star-type operator.  We define $*_C: \tC^k(\om) \rightarrow \tC^{2n+1-k}(\om)$ by
\begin{align}\label{Cstar}
*_C \left(\eta_k + \theta \xi_{k-1}\right) = *\xi_{k-1} + \theta\, (-1)^k *\eta_k\,.    
\end{align}
We can then write
\begin{align*}
\langle\eta_k + \theta \xi_{k-1}, \eta'_k + \theta \xi'_{k-1}\rangle_C 
&=\int_M \dfrac{\partial}{\partial\theta} \left(\left(\eta_k + \theta \xi_{k-1}\right) \w *_C \left(\eta'_k + \theta \xi'_{k-1}\right)\right)\\
&=\int_M \eta_k \w * \eta'_k \, + \, \xi_{k-1} \w * \xi'_{k-1}
\end{align*}
where the derivative  $(\partial/\partial\theta)$ satisfies  $(\partial/\partial\theta)(\theta\eta_k) =\eta_k$ for any $\eta_k\in \Om^k(M)$.  (For ease of notation, when it is clear that we are considering the inner product for cone forms, we will simply write $\langle\  ,\  \rangle$ to denote the cone inner product of \eqref{cinprod} and leave out the $C$ subscript.)

Turning to the differential operators acting on $\tC^*(\om)$, the differential, $d_C: \tC^k(\om)\to \tC^{k+1}(\om)$ is defined to be  
\begin{align*}
d_C (\eta_k+\theta  \xi_{k-1})=d\eta_k+\omega \wedge \xi_{k-1}-\theta  d\xi_{k-1}\,,    
\end{align*}
which corresponds simply to the exterior derivative acting on $(\eta_k+\theta  \xi_{k-1})$ and using $d\theta = \om$.  Clearly, $d_C\, d_C =0\,$.
And with respect to the cone inner product \eqref{cinprod}, the adjoint of $d_C$ has the form
\begin{align}
d_C^*(\eta_{k}+\theta\xi_{k-1})&= d^*\eta_k+\theta(\Lambda \eta_k-d^*\xi_{k-1})\nonumber\\
&=(-1)^{k}*_C d_C *_C (\eta_{k}+\theta\xi_{k-1})\label{dsCdef}
\end{align}
where $\Lambda=\om^*$ denotes the adjoint of $\om$ with respect to the inner product \eqref{Minprod} on $M$  and has the expression $\Lambda =(-1)^k * \omega\, *\,$ when acting on a $k$-form. 
For convenience, we will often express the cone form as a two-vector, $\sigma_k=\begin{pmatrix}\eta_k \\ \xi_{k-1}\end{pmatrix} \in \tC^k(\omega)$.  In this notation, $d_C$ and its adjoint $d_C^*$ have the following matrix form
\begin{align*}
d_C=\begin{pmatrix} d & \omega \\ 0 & -d \end{pmatrix}\,, \qquad d_C^*=\begin{pmatrix} d^* & 0 \\ \Lambda & -d^* \end{pmatrix}\,.
\end{align*}

As in Witten \cite{Witten}, we deform the above differential operators by a Morse function $f$ parameterized by a non-negative real number $t\in \mathbb{R}^+$.  First, the deformed exterior derivative and its adjoint take the form
\begin{align}\label{dtdts}
d_t = e^{-tf}d\; e^{tf} = d + t\, df\,, \qquad \quad
d_t^* = e^{tf}d^*\, e^{-tf} = d^* + t\,\iota_{\nabla f}\,.
\end{align}
As for the cone differentials $(d_C, d_C^*)$, their deformation can be expressed simply in terms of $(d_t, d_t^*)$ as follows, 
\begin{align}\label{dcdcs}
d_{C, t}=e^{-tf}d_C\:e^{tf}
=\begin{pmatrix} d_t & \omega \\ 0 & -d_t \end{pmatrix}\,, \qquad
d_{C, t}^*=e^{tf}d_{C}^*\:e^{-tf}
=\begin{pmatrix} d_t^* & 0 \\ \Lambda & -d_t^* \end{pmatrix}.
\end{align}

Let us now turn our attention to the Laplacian operator associated to the de Rham cone complex:
\begin{align}\label{DeltaC}
\Delta_C = d_C d_C^* + d_C^* d_C = \begin{pmatrix}  \Delta + \om \Lambda & -d\Lambda + \Lambda d  \\ -\om d^* + d^*\om & \Delta + \Lambda \om \end{pmatrix},
\end{align}
where $\Delta= dd^* + d^*d$ is the standard de Rham Laplacian.    Under deformation by a Morse function $f$, it becomes 
\begin{align}
\Delta_{C,t} &= d_{C,t}d^*_{C,t}+ d^*_{C,t}d_{C,t}
=\begin{pmatrix} \Delta_t+\omega \Lambda & -d_t^{\Lambda*} \\ -d_t^\Lambda & \Delta_t+\Lambda\omega  \end{pmatrix} ,\label{DeltaCt}
\end{align} 
where 
\begin{align*}
\Delta_t&= d_t d_t^* + d_t^* d_t = \Delta + t (\mathcal{L}_{\nabla f}+\mathcal{L}^*_{\nabla f}) + t^2 ||df||_x^2, \\
d_t^\Lambda &= d_t\Lambda-\Lambda d_t
=d\Lambda-\Lambda d+t(df\Lambda - \Lambda df)=d^\Lambda+t(df\Lambda - \Lambda df),  \\
d_t^{\Lambda*} &= \omega d_t^*-d_t^*\omega
=\omega d^*-d^*\omega+t(\omega\,\iota_{\nabla f} - \iota_{\nabla f}\,\omega)=d^{\Lambda*}+t(\omega\iota_{\nabla f} - \iota_{\nabla f}\omega). 
\end{align*}
In the above calculation, we used the notation $d^\Lambda:=d\Lambda - \Lambda d\,$, and $d^{\Lambda*}$ is its adjoint.  Furthermore, $\mathcal{L}_{\nabla f}$ is the Lie derivative with respect to the gradient of $f$, with  $\mathcal{L}^*_{\nabla f}$ being its adjoint, and $||df||_x^2=g^{ij}\partial_if\partial_jf$ is the pointwise norm of $df$ with respect to the Hodge metric on forms.   We will call $\Delta_C$ the cone Laplacian and $\Delta_{C,t}$ the deformed cone Laplacian.

It turns out to be useful to also consider deforming the differential operators with respect to $(-f)$, which is also a Morse function.  We will denote operators deformed by $(-f)$ with a $(-t)$ subscript.  For instance,  
\begin{align*}d_{C, -t}:=e^{-t(-f)}d_C\,e^{t(-f)}
=\begin{pmatrix} d_{-t} & \omega \\ 0 & -d_{-t} \end{pmatrix}\,, \qquad
d_{C, -t}^*:=e^{t(-f)}d_{C}^*\,e^{-t(-f)}
=\begin{pmatrix} d_{-t}^* & 0 \\ \Lambda & -d_{-t}^* \end{pmatrix}.
\end{align*}
and similarly, $\Delta_{C, -t}:= d_{C, -t} d_{C, -t}^*  + d_{C, -t}^* d_{C, -t}\, $.  Noting from \eqref{Cstar} that $*_C *_C = Id$ and from \eqref{dsCdef} that $d_C^* = (-1)^k *_C d_C \: *_C$ acting on $\sigma_k\in \tC^k(\om)$, we find the following relations:
\begin{align}\label{ntrels}
d_{C, -t} = (-1)^{k+1} *_C d^*_{C,t}\, *_C\,, \qquad
d^*_{C,-t} = (-1)^k*_C d_{C,t}\:*_C\,,
\end{align}
acting on $\sigma_k\in \tC^k(\om)$. 
These immediately imply the following:
\begin{lemma}\label{2.1}
The cone form  $\sigma=\begin{pmatrix} \eta_k \\ \xi_{k-1} \end{pmatrix}\in \tC^k(\omega)$ is a harmonic solution of $\Delta_{C, -t}$ for the Morse function $(-f)$, if and only if $*_C\,\sigma=\begin{pmatrix} *\xi_{k-1} \\ (-1)^k*\eta_k \end{pmatrix}\in \tC^{2n+1-k}(\om)$ is a harmonic solution of $\Delta_{C, t}$ for the Morse function $f$.  In particular, the harmonic conditions $d_{C, -t}\,\sigma=0$ and $d^*_{C,-t}\,\sigma=0$ hold if and only if $d_{C, t}(*_C\,\sigma)=0$ and $d^*_{C, t}(*_C\,\sigma)=0\,$.  
\end{lemma}
The relation also extends to all eigenforms of the deformed cone Laplacian.
\begin{lemma}\label{2.2}
The cone form $\sigma\in \tC^k(\om)$  is an eigenform of $\Delta_{C, -t}\,$ if and only if $(*_C
\, \sigma_k)\in \tC^{2n+1-k}(\om)$ is an eigenform of $\Delta_{C, t}\,.$
\end{lemma}
\begin{proof}
The statement follows from \eqref{ntrels}.  Specifically, acting on $\sigma_k\in \tC^k(\om)$, we have
\begin{align*}
\Delta_{C, -t}&= d_{C, -t} d_{C, -t}^*  + d_{C, -t}^* d_{C, -t} \\
&=\left[(-1)^k*_Cd^*_{C,t}*_C\right]\left[(-1)^k *_C d_{C,t}*_C\right]+\left[(-1)^{k-1}*_Cd_{C,t}*_C\right]\left[(-1)^{k+1}*_Cd^*_{C,t}*_C\right]\\
&=*_C\, \Delta_{C,t}\,*_C
\end{align*}
\end{proof}

\subsection{Local harmonic solutions of the deformed cone Laplacian}\label{harmsol}

We are interested in studying the spectrum of the deformed cone Laplacian $\Delta_{C,t}$ when $t$ is large.  Following Witten's observation in \cite{Witten}, as $t\to \infty$, the order $t^2$ term of $\Delta_{C,t}$ in \eqref{DeltaCt}, $t^2 ||df||_x^2\,$, dominates, and so eigenforms of $\Delta_{C,t}$ must localize around the critical points of the Morse function $f$.  This localization greatly simplifies the study of the $\Delta_{C,t}$ spectrum and it turns out understanding the local eigenforms of $\Delta_{C,t}$ in a local chart around a critical point of $f$ is sufficient for obtaining the Morse-type inequalities for $H(\tC(\om))$.  In this subsection, we will write down the local harmonic solutions for $\Delta_{C,t}$ for large $t$ around critical points of $f$.  The local non-harmonic eigenforms will be taken up in the next subsection. 

For studying the eigenforms of $\Delta_{C,t}$ in a local neighborhood around $p\in Crit(f)$, it is useful to work in a local coordinate chart where $(\om, g, f)$ all have standard canonical forms. For this, Stratmann \cite{Stratmann} showed that in the neighborhood of any $p\in Crit(f)$, the pullback of $f$ under a properly chosen symplectomorphism can be expressed in the following form, 
$$f=\displaystyle f(p)+\sum_{\ell=1}^{n_f(p)} -x_\ell^2/2\; +\sum_{\ell=n_f(p)+1}^{2n} x_\ell^2/2.
$$ 
where $n_f(p)$  denotes the index at the critical point $p$.  (We shall often use the notation $n_f(p)$ instead of $ind(p)$ to emphasize the dependence of the index on $f$.)  In other words, it is possible to modify the Morse function via a series of local symplectomorphisms if needed so that around any $p\in Crit (f)$ we have what we shall call a compatible coordinate chart.
\begin{definition}\label{Compcoord}
A local coordinate chart $\{x_j\}_{j=1, \dots,2n}$ around a critical point $p\in Crit(f)$ is called a {\bf compatible coordinate chart} if the following properties are satisfied simultaneously: \begin{itemize}
\item Darboux coordinates, i.e. $\omega = \displaystyle\sum_{i=1}^n dx_i \wedge dx_{i+n}\,$;
\item Normal coordinates, i.e. $g_{ij}(x)=\delta_{ij}+O(|x|^2)\,$;
\item Morse coordinates, $f =f(p)+\displaystyle\sum_{\ell=1}^{n_f(p)} -x_\ell^2/2\;+\sum_{\ell=n_f(p)+1}^{2n} x_\ell^2/2$. 
\end{itemize}
\end{definition}

\medskip

In the following, when discussing localized eigenform solutions of $\Delta_{C,t}\,$, we will by default work in such a compatible coordinate chart.

For local harmonic solutions $\sigma_k = \begin{pmatrix} \eta_k\\ \xi_{k-1}\end{pmatrix}\in \tC^k(\om)$, the harmonic conditions  $d_{C, t}\,\sigma_k=d^*_{C,t}\,\sigma_k=0\,$, with $(d_{C,t}, d^*_{C,t})$ given in  \eqref{dcdcs}, impose the following four conditions:
\begin{alignat}{3}\label{dctcond}
d_{C, t}\,\sigma_k&=0:\qquad\qquad  &&(a)
\ d_t\eta_k+\omega \wedge \xi_{k-1} = 0\,, \qquad\quad&&
(b)\ d_t\xi_{k-1}=0\,;
\\
\label{dsctcond}
d^*_{C,t}\,\sigma_k&=0:\qquad
&&(c) \ d_t^*\eta_k=0 \,,  
&&
(d)\ d_t^*\xi_{k-1}-\Lambda \eta_k =0\,.
\end{alignat}
Notice that if both $\om$ and $\Lambda$ were set to zero, then the above four conditions become just $d_t\eta_k = d^*_t\eta_k =0$, and $d_t \xi_{k-1}=d^*_t\xi_{k-1} =0$, which are just the usual deformed harmonic conditions.
Such solutions were described by Witten \cite{Witten} which we recall here.
\begin{lemma}[Witten \cite{Witten}]\label{WittenH}
Around a critical point $p\in Crit(f)$ with index $n_f(p)=k$ described by a local compatible  coordinate chart  $\{x_i\}_{i=1,\ldots,2n}$\,, 
such that $g_{ij}=\delta_{ij}$ and 
\begin{align*}
f(x) = f(p) + \frac{1}{2}\left(-x^2_1 - \ldots - x^2_k +x^2_{k+1}+\ldots x^2_{2n}\right),
\end{align*}
there is a one-dimensional $k$-form solution 
generated by 
\begin{align}\label{Wsol}
\zeta_k:=e^{-t|x|^2/2}dx_1\wedge \ldots \wedge dx_k\,,   
\end{align}
that satisfies the deformed harmonic conditions, i.e. 
$d_t\zeta_k = d^*_t \zeta_k =0\,$.
\end{lemma}
Hence, in the case where $\om=0$, we would have two types of localized harmonic solutions for $\sigma_k\in \tC^k(\om)\,$, generated by $(\zeta_{k}, 0)$ at all $p\in Crit(f)$ with $n_f(p)=k$, and also, $(0, \zeta_{k-1})$ at all $q\in Crit(f)$ with $n_f(q)=k-1\,.$
It turns out that the existence of two types of harmonic generators for each cone degree $k$ persists even when $\om\neq
0\,.$
\begin{prop}[Local harmonic solutions of the deformed cone Laplacian] 
\label{prop:harmonicsmallk} 
For $\tC^k(\om)$, there exist two types of local harmonic solutions of the deformed cone Laplacian $\Delta_{C,t}\,$: localized about critical points $p\in Crit(f)$ with index $n_f(p)=k$ and  localized about critical points $q\in Crit(f)$ with index $n_f(q)=k-1$.  Each index $k$ or $k-1$ critical point has one generating harmonic solution. 
\end{prop}
\begin{proof}
The deformed cone harmonic solutions must satisfy the four conditions in \eqref{dctcond}-\eqref{dsctcond}.
We will describe the solutions first in the case for cone forms of degree $k\leq n$ and then for the case of degree $k\geq n+1$.  

\medskip

\noindent {\bf Case (I):   $0\leq k\leq n$} 

When $k\leq n$, the harmonic solutions of $\Delta_{C,t}$ localized around critical points $p\in Crit(f)$ with index $n_f(p)=k$ are the standard type in \eqref{Wsol}.  Expressed in compatible coordinates they are generated by 
\begin{align}\label{kln1}
\begin{pmatrix} \eta_k \\ \xi_{k-1} \end{pmatrix}=\begin{pmatrix} \zeta_{k} \\ 0 \end{pmatrix}=\begin{pmatrix} e^{-t|x|^2/2}dx_1\wedge...\wedge dx_{k} \\0 \end{pmatrix}\,.
\end{align}
Since $d_t\,\zeta_{k}=d_t^*\,\zeta_{k}=0$, it remains only to check that $\Lambda\, \zeta_{k} = 0$.  For this, we note that in a compatible coordinate chart, $\Lambda$ takes the simple form
\begin{align}\label{Llocal}
\Lambda=\sum_{i=1}^{n} \iota_{\partial_{x_{i+n}}}\iota_{\partial_{x_i}}\,.
\end{align}
Clearly then, we have $\Lambda\, \zeta_{k} = 0$ as long as  $k\leq n\,$.

We describe now the localized harmonic solutions of $\Delta_{C,t}$ around critical points $q\in Crit(f)$  with index $n_f(p)=k-1$.  In the compatible coordinate chart, the Morse function takes the form
\begin{align}\label{fkmo}
f(x)=f(q)+\frac{1} {2} ( -x_1^2-x_2^2-...-x_{k-1}^2+x_k^2+...+x_{2n}^2)\,,
\end{align}  
and we also introduce a local one form
\begin{align*}
\tau =\displaystyle \sum_{i=1}^{k-1} -x_{i+n}dx_i+\frac{1} {2} \sum_{i=k}^n (x_{i}dx_{i+n}-x_{i+n}dx_i)\,,
\end{align*}
with the property $d\tau=\omega$. Then the following generates harmonic solutions around critical points $q$ with index $n_f(q)=k-1$:
\begin{align}\label{kln2}
\begin{pmatrix} \eta_k \\ \xi_{k-1} \end{pmatrix}=\begin{pmatrix} -\tau \wedge\zeta_{k-1} \\ \zeta_{k-1} \end{pmatrix}=\begin{pmatrix}
\frac{1}{2} e^{-t |x|^2/2} \displaystyle\sum_{i=k}^n \left(x_{i+n} dx_i- x_i  dx_{i+n} \right) \wedge dx_1 \wedge \ldots \wedge dx_{k-1} \\ e^{-t|x|^2/2}dx_1\wedge\ldots\wedge dx_{k-1} \end{pmatrix}.
\end{align}
Let us check conditions $(a)$-$(d)$ in \eqref{dctcond}-\eqref{dsctcond}.
For (a), it follows by direct computation that 
\begin{align*} 
d_t\eta_k=d_t(-\tau \wedge \zeta_{k-1}) 
 =-d\tau \wedge \zeta_{k-1}+\tau \wedge d_t\zeta_{k-1} 
 =-\omega \wedge \xi_{k-1}\,.
\end{align*}
Condition (b) is trivially satisfied since $d_t\zeta_{k-1}=\,0$. For (c), we use the expression $d^*_t=e^{tf} d^* e^{-tf}$ from \eqref{dtdts} and compute also making use of \eqref{fkmo}: 
\begin{align*}
d_t^*\eta_k&=e^{tf}d^*(e^{-tf} \eta_k) \\
&=\frac{1}{2}e^{tf}d^*\left(e^{-tf}e^{-t |x|^2/2} \sum_{i=k}^n \left( x_{i+n} dx_i-x_i  dx_{i+n}\right) \wedge dx_1 \wedge \ldots \wedge dx_{k-1} \right) \\
&=\frac{1}{2}e^{t(f-f(q))}d^*\left(\exp\Big(-t \sum_{\ell=k}^{2n} x_\ell^2\Big)
\sum_{i=k}^{n} \left(x_{i+n} dx_i-x_i dx_{i+n} \right) \wedge dx_1 \wedge \ldots \wedge dx_{k-1} \right)\\
&=e^{t(f-f(q))}\exp\Big(-t \sum_{\ell=k}^nx_\ell^2\Big)\left(t\sum_{i=k}^n\left( x_ix_{i+n}-x_{i+n}x_i\right)\right) dx_1 \wedge \ldots \wedge dx_{k-1}\ =0\,.
\end{align*} 
And for condition (d), since $d^*_t \xi_{k-1}= d^*_t\zeta_{k-1}=0\,$, we only need to check that $\Lambda \eta_k =0$, which follows from using the expression of $\Lambda$ in \eqref{Llocal}.

\medskip

\noindent{\bf Case (II): $n+1\leq k\leq 2n+1$}

Using Lemma \ref{2.1}, the generators of the harmonic solutions for $k \geq n+1$ can be straightforwardly obtained by applying $*_C$ to the generators found for $k\leq n$ in Case (I).  Specifically, let $\tsigma_j \in \tC^j(\om)$ with $j\leq n$ be a harmonic solution of the deformed cone Laplacian with respect to a Morse function $(-f)$, i.e. $d_{C,-t}\tsigma_j =d_{C,-t}^*\tsigma_j=0\,$. Then, by Lemma \ref{2.1},  $*_C\tsigma_j\in\tC^{k}(\om)$ with degree $k=2n+1-j\geq n+1$ is a harmonic solution with respect to $f$, i.e.  $d_{C,t}(*_C\tsigma_j)  =d_{C,t}^*(*_C\tsigma_j)=0\,$. Thus, if we take $\tsigma_j$ to be the two types of local harmonic generators  \eqref{kln1} and \eqref{kln2} found for cone forms of degree $j\leq n\,$ with respect to a Morse function which we label by $-f$, then  applying $*_C$ to them would result in two types of harmonic generators for cone forms of degree $k=2n+1-j$.  And since $j=0, 1,\ldots, n$, this means $k= n+1, \dots, 2n+1,$ as desired. 

Let us comment about the minus sign difference in the Morse function for $\tsigma_j$, which is associated with $(-f)$, versus $\sigma_k=*_C\,\tsigma_j\,$, associated with $f$.  Notice first that a critical point $p\in Crit(-f)$ of index $n_{-f}(p)$ would remain a critical point $p\in Crit(f)$ but with index $n_{f}(p)=2n-n_{-f}(p)$.  For in the compatible coordinate chart, we have
\begin{align*}
f&= f(p)+\frac{1}{2}(-x_1^2-\ldots -x_{n_f(p)}^2+x_{n_f(p)+1}^2 + \ldots +  x_{2n}^2)\\
-f &= -f(p)+\frac{1}{2}(x_1^2+\ldots +x_{n_f(p)}^2-x_{n_f(p)+1}^2 - \ldots - x_{2n}^2)\\
&= -f(p) + \frac{1}{2}(-\tx_1^2 -\ldots-\tx_{n_{-f}(p)}^2 + \tx_{n_{-f}(p)+1}^2 + \ldots + \tx^2_{2n})
\end{align*}
where in the last line, we have applied a change of coordinates $x_i \to \pm\,\tx_{2n+1-i}$ to return to compatible coordinate chart.  (We note that the constant $f(p)$ is inconsequential in our discussion here since the deformed harmonic differentials $(d_{C,t}, d^*_{C,t})$, given in \eqref{dcdcs}, do not depend on it.) 
In particular, the two types of harmonic generators localized around index $j$ critical points \eqref{kln1}
and around index $j-1$  critical points \eqref{kln2}
of $(-f)$ will still be localized at the same critical points after applying $*_C$.  But the index of the critical points would become $2n-j=k-1$ and $2n-(j-1)=k$, respectively, defined with respect to $f$.

We will write down explicitly the harmonic generators obtained from applying the $*_C$ map.  
Applying $*_C$ first to the generator in \eqref{kln1}, we obtain the local harmonic solutions about any critical point $q$ of index $n_f(q)=k-1$: 
\begin{align}\label{kgn1}
\begin{pmatrix} \eta_k \\ \xi_{k-1} \end{pmatrix}=\begin{pmatrix} 0 \\ \zeta_{k-1} \end{pmatrix}=\begin{pmatrix} 0 \\ e^{-t|x|^2/2}dx_1\wedge\ldots\wedge dx_{k-1} \end{pmatrix}.
\end{align}
Applying $*_C$ to the generator given in \eqref{kln2}, we obtain a second localized generator this time around any critical point $p$ of index $n_f(p)=k$:
\begin{align}\label{kgn2}
\begin{pmatrix} \eta_k \\ \xi_{k-1} \end{pmatrix}&=\begin{pmatrix} \zeta_{k} \\ \iota_Z\zeta_{k} \end{pmatrix}
=\begin{pmatrix} e^{-t|x|^2/2} dx_1\wedge\ldots\wedge dx_{k} \\  \frac{1}{2} e^{-t|x|^2/2}\dfrac{\omega^{r-1}} {(r-1)!}\wedge\displaystyle\sum_{i=1}^{r} \left(x_idx_i +x_{n+i}dx_{n+i}\right)\wedge dx_{r+1} \wedge \ldots \wedge dx_{n} \end{pmatrix},
\end{align}
where $r=k-n$.  In \eqref{kgn2}, we have noted that the second component can be expressed in terms of an interior product by a vector which we have denoted by $Z$.  To define this vector, note that 
by \eqref{Cstar} and \eqref{kln2}, the second component is given by
\begin{align*}
\xi_{k-1}&=(-1)^k*(-\tau \wedge e^{-t|x|^2/2}dx_{2n}\,\wedge\ldots\wedge dx_{2n+1-k})\\
&=\iota_Z\left(e^{-t|x|^2/2}dx_1\wedge\ldots \wedge dx_{k-1}\right)
\end{align*}
where $Z=-\tau^{\prime \sharp}$ is the musical isomorphism applied to the one form 
\begin{align*}\tau^\prime&=\displaystyle\sum_{i=0}^{2n-k}x_{n-i}dx_{2n-i}-\frac{1} {2}\sum_{i=2n-k+1}^{n}x_{2n-i}dx_{n-i}-x_{n-i}dx_{2n-i} \\
\text{so } Z &= \sum_{i=0}^{2n-k}- x_{n-i}\frac{\partial} {\partial x_{2n-i}}+\frac{1} {2}  \sum_{i=2n-k+1}^{n} x_{2n-i}\frac{\partial} {\partial x_{n-i}}-x_{n-i}\frac{\partial} {\partial x_{2n-i}}
\end{align*}
where $\tau^\prime$ is $\tau$ under the symplectomorphism that maps $x_{2n+1-i} \to \pm\,x_i$ (to account for the sign change from $-f$ to $f$). 

Finally, let us add that it can be checked directly that the above two harmonic generators for $k\geq n+1$ \eqref{kgn1}-\eqref{kgn2} satisfy the four deformed harmonic conditions (a)-(d) in \eqref{dctcond}-\eqref{dsctcond} in a compatible coordinate chart. 
\end{proof}

Proposition \ref{prop:harmonicsmallk} thus tells us that for each cone degree $k$, there exist at least two types of generators of local harmonic solutions of the deformed cone Laplacian.  (In the next subsection, we will show that any other local eigenforms will have a non-zero eigenvalue and hence not harmonic.)  Since the two types of harmonic generators are localized at critical points of index $k$ and $k-1$, this implies each critical point $p\in Crit(f)$ of index $n_f(p)$ support two harmonic generators of cone degree $n_f(p)$ and $n_f(p)+1$.  Below, we express this observation as a corollary.

\begin{corollary}\label{gen2sol}
Let $f:M \to \mathbb{R}$ be a Morse function on a symplectic manifold $(M^{2n}, \omega)$. Then in the compatible coordinate chart around $p\in Crit(f)$ with index $n_f(p)=k$, where
\begin{align*}
\omega = \sum_{i=1}^n dx_i \wedge dx_{i+n}\,,\qquad
g_{ij}=\delta_{ij}\,,\qquad 
f=f(p) + \frac{1}{2}\left(-x^2_1 - \ldots - x^2_k +x^2_{k+1}+\ldots x^2_{2n}\right)\,,
\end{align*}
we have the following two generators of harmonic solutions to the deformed cone Laplacian:
\begin{align}\label{hsatk}
\begin{pmatrix} \zeta_{k} \\ \iota_Z\zeta_{k} \end{pmatrix}\in \tC^k(\om)\,, \qquad
\begin{pmatrix} -\tau \wedge \zeta_{k} \\ \zeta_{k}\end{pmatrix}\in \tC^{k+1}(\om)\,,
\end{align}
where $\zeta_k = e^{-t|x|^2/2} dx_1 \w\ldots\w dx_k$, and 
\begin{align} 
\tau&=
\displaystyle\sum_{i=1}^{n_f(p)}-x_{i+n}dx_i+\frac{1} {2} \sum_{i=n_f(p)+1}^{n}x_idx_{i+n}-x_{i+n}dx_i\,, \quad 0\leq n_f(p) \leq n-1
 \label{taudef}\\
\iota_{Z}&=
\displaystyle\sum_{i=0}^{2n-n_f(p)} -x_{n-i}\iota_{\partial_{2n-i}}+\frac{1} {2} \sum_{i=2n-n_f(p)+1}^n x_{2n-i}\iota_{\partial_{n-i}}-x_{n-1}\iota_{\partial_{2n-i}}\,,\quad n+1\leq n_f(p)\leq 2n
\label{izdef}
\end{align}
and defining $\tau=0$ when $n_f(p)\geq n\,$, and $\iota_Z=0$ when $n_f(p)\leq n$.
\end{corollary} 

\subsection{Bounding eigenvalues of local eigenforms of the deformed cone Laplacian}
Let $\begin{pmatrix} \eta_k \\ \xi_{k-1} \end{pmatrix}$ be any local eigenform of $\Delta_{C,t}$ with eigenvalue $\lambda$ and not generated by the harmonic generators of the previous subsection. We wish to show that $\lambda>c\,t$ for some positive, non-zero constant $c$. Therefore, as $t$ goes to infinity, $\lambda$ grows to infinity at least of the order $c\,t$. 
\begin{theorem}
On a compatible local coordinate chart near a critical point $p$ of a Morse function $f$, suppose $\begin{pmatrix} \eta_k \\ \xi_{k-1} \end{pmatrix}$ is orthogonal to the local harmonic solutions generated by \eqref{hsatk}. Then the following inequality holds,  \begin{equation}\label{ineq:DeltatC}
  \left\langle \Delta_{C,t}\begin{pmatrix} \eta_k \\ \xi_{k-1}\end{pmatrix}, \begin{pmatrix} \eta_k \\ \xi_{k-1}\end{pmatrix} \right\rangle \geq c\,t\left|\begin{vmatrix}  \eta_k \\ \xi_{k-1}\end{vmatrix}\right|^2+O(\|\eta_k\|^2, \|\xi_{k-1}\|^2)\,,
\end{equation}
where $c$ is a positive constant.  In particular, the eigenvalue of any local non-harmonic eigenform is greater than $c'\,t$ for some fixed constant $c'>0\,$. 
\end{theorem}
\begin{proof}
We start with the following computation,
\begin{align}
&\left\langle \Delta_{C,t}
\begin{pmatrix} \eta_k \\ \xi_{k-1}\end{pmatrix}, 
\begin{pmatrix} \eta_k \\ \xi_{k-1}\end{pmatrix} \right\rangle =\left\langle 
\begin{pmatrix} \Delta_{t}+\omega \Lambda & -d_t^{\Lambda*} \\ -d_t^{\Lambda} & \Delta_{t}+\Lambda\omega  \end{pmatrix}
\begin{pmatrix} \eta_k \\ \xi_{k-1}\end{pmatrix}, 
\begin{pmatrix} \eta_k \\ \xi_{k-1}\end{pmatrix}
\right\rangle \nonumber\\
  &=\left\langle \begin{pmatrix} \Delta_{t}\eta_k+\omega \Lambda\eta_k -d_t^{\Lambda*}\xi_{k-1} \\ -d_t^{\Lambda}\eta_k+ \Delta_{t}\xi_{k-1}+\Lambda\omega \xi_{k-1}  \end{pmatrix}, \begin{pmatrix} \eta_k \\ \xi_{k-1}\end{pmatrix} \right\rangle \nonumber\\
  &=\langle \Delta_{t}\eta_k, \eta_k \rangle \!+\!\langle \omega \Lambda\eta_k, \eta_k \rangle \!-\!\langle d_t^{\Lambda*}\xi_{k-1}, \eta_k \rangle
  \!+\! \langle\Delta_{t}\xi_{k-1}, \xi_{k-1} \rangle \!-\!\langle d_t^{\Lambda}\eta_k, \xi_{k-1} \rangle\!+\!\langle \Lambda\omega \xi_{k-1} ,  \xi_{k-1} \rangle \nonumber\\
  &=\langle \Delta_{t}\eta_k, \eta_k \rangle  -2\langle d_t^{\Lambda}\eta_k, \xi_{k-1}\rangle+ \langle\Delta_{t}\xi_{k-1}, \xi_{k-1} \rangle
  +\|\Lambda\eta_k\|^2+\|\omega \xi_{k-1}\|^2. \label{eq:DeltatCformula}
\end{align}
  
Note that $\Delta_{t}$ is self-adjoint and has a basis of orthogonal eigenforms of the form $ \left(\displaystyle\prod_j H_{m_j}(\sqrt{t}x_j)\right)e^{-t|x|^2/2}dx_{I_k}$ where $H_{m_j}$ is the $m_j$-th Hermite polynomial. The corresponding eigenvalues are $2t(\ell+\sum m_j)=2t(\ell+m)$ where $\ell$ is the number of missing/different coordinates in $dx_{I_k}=dx_{i_1}\wedge...\wedge dx_{i_k}$ from harmonic form's $dx_1\wedge...\wedge dx_{n_f}$ and $\sum m_j =m$ is the sum of the Hermite polynomial numbers. Imposing the $L^2$-norm condition on $\eta_k$ and $\xi_{k-1}$, we can express them as linear combinations of these eigenforms.  For instance, we shall write  
\begin{align}
   \eta_k=\sum_{\alpha}\eta^\alpha =\sum _\alpha a_\alpha\prod_j  H_{m_j}(\sqrt{t}x_j)e^{-t|x|^2/2} dx_{I_k}\,, \label{etaexp}
\end{align}
where the index $\alpha$ denotes a certain combination of $\{I_k, \{m_j\}\}$, and similarly,  
\begin{align}
   \xi_{k-1}=\sum_\beta \xi^\beta =\sum _\beta b_\beta\prod_jH_{m_j}(\sqrt{t}x_j)e^{-t|x|^2/2}dx_{I_{k-1}}\,, \label{xiexp}
\end{align}
with the degree of the form adjusted accordingly.  Furthermore, we will refer to all  
$\eta^\alpha$'s and $\xi^\beta$'s eigenforms
as simple forms. 

Below, we will prove the theorem assuming  $k \leq n$. For $k > n$, we can use the observation of Lemma \ref{2.2} that $\Delta_{C,-t}=*_C\,\Delta_{C,t} \,*_C$.  Under $f \to -f$, we have the relation $n_f(p) = 2n -n_{-f}(p)$  and so a similar proof would follow. 

As  $\Delta_{C,t}$ is self-adjoint, we have that if $\begin{pmatrix} \eta_k \\ \xi_{k-1} \end{pmatrix}$ is orthogonal to both $ \begin{pmatrix} \zeta_k \\ 0 \end{pmatrix}$ and $ \begin{pmatrix}  -\tau \wedge \zeta_{k-1} \\ \zeta_{k-1}  \end{pmatrix}$, $$\left\langle \begin{pmatrix} \eta_k \\ \xi_{k-1} \end{pmatrix},  \begin{pmatrix} \zeta_k \\ 0 \end{pmatrix}\right\rangle=0=\left\langle\begin{pmatrix} \eta_k \\ \xi_{k-1} \end{pmatrix}, \begin{pmatrix} -\tau \wedge \zeta_{k-1} \\ \zeta_{k-1} \end{pmatrix}  \right\rangle$$ where $\zeta_{k},\zeta_{k-1}$ are solutions of $\Delta_{t}$. We notice that $\eta_k$ cannot have a nonzero component in the kernel of $\Delta_{t}$, as then it would not be orthogonal to $\begin{pmatrix} \zeta_{k} \\ 0\end{pmatrix}$. Note that $\xi_{k-1}$ may have a component $b_\beta\zeta_{k-1}$ in the expansion of \eqref{xiexp}.

If $\xi$ has a component of $\xi^\beta=b_\beta\zeta_{k-1}$, then we must have
\begin{align*}
0&=\left\langle\begin{pmatrix} \eta_k \\ \xi_{k-1}\end{pmatrix}, \begin{pmatrix} -\tau \wedge \zeta_{k-1} \\  \zeta_{k-1} \end{pmatrix}\right\rangle 
 \\ &
= \langle \eta_k, -\tau \wedge \zeta_{k-1}\rangle+b_\beta\left\langle \zeta_{k-1} ,\zeta_{k-1} \right\rangle 
=\langle \eta_k, -\tau \wedge \zeta_{k-1}\rangle+b_\beta\|\zeta_{k-1}\|^2. 
\end{align*}
Thus, there must be  special components of $\eta$ in the expansion \eqref{etaexp}, which we will label together by $\eta^{\alpha'}$, taking the following form  
\begin{align}\label{ortoHarm} 
\eta^{\alpha^\prime}=\sum_j(a_{j}x_{n+j}dx_j-a_{n+j}x_jdx_{n+j})\wedge\zeta_{k-1}
\textbf{ and } b_\beta\|\zeta_{k-1}\|^2=-\langle \eta^{\alpha^\prime}, -\tau \wedge \zeta_{k-1}\rangle.
\end{align} 
Note that by normalization, we have 
\begin{align*}
\|\zeta_{k-1}\|^2=\int e^{-t|x|^2}d{V\!ol}=\frac{\pi^{n}} {t^{n}}\,, 
\end{align*} 
\begin{align}
-\langle \eta^{\alpha^\prime},-\tau \wedge \zeta_{k-1} \rangle=-\int \sum \frac{a_j}{2} x_j^2e^{-t|x|^2}d{V\!ol} =-\frac{\pi^n\sum a_j} {4t^{n+1}}\,. \label{taunorm} 
\end{align}
Thus, 
\begin{align}
\label{ajsum}b_\beta =-\frac{\sum a_j} {4t}. 
\end{align}

Next we calculate $d^\Lambda_t\eta^{\alpha^\prime}=(d_t\Lambda-\Lambda d_t)\eta^{\alpha^\prime}$. Note that a particular form 
\[
\chi_j=x_{n+j}e^{-t|x|^2/2}dx_j\wedge dx_1\wedge...\wedge dx_{n_f}
\]
has no components containing the wedge pair $dx_r\wedge dx_{r+n}$ for any $1\leq r\leq n$, as $n_f\leq n$. Using the local form of $\Lambda$ given in \eqref{Llocal}, we have $\Lambda\chi_j=0$ and can compute $d^\Lambda_t\chi_j=(d_t\Lambda-\Lambda d_t)\chi_j$ as follows, 
  \begin{align*}
  (d_t\Lambda-\Lambda d_t)\chi_j&=-\Lambda d_t(t x_{n+j}dx_{j}\zeta_{k-1}) \\
  &=-\Lambda (d+tdf\wedge)( x_{n+j}dx_{j}\wedge\zeta_{k-1}) \\
  &=-\Lambda \Big(d(x_{n+j}dx_{j})\wedge \zeta_{k-1}-x_{n+j}dx_{j}\wedge d\zeta_{k-1}- x_{n+j}dx_{j}(tdf\wedge \zeta_{k-1})\Big) \\
  &=-\Lambda \Big(-dx_{j}\wedge dx_{n+j}\wedge \zeta_{k-1}-x_{j}dx_{n+j}\wedge d_t\zeta_{k-1}\Big) \\
  &=-\Lambda (-dx_{j}\wedge dx_{n+j}\wedge \zeta_{k-1}-x_{j}dx_{n+j}\wedge 0).
  \end{align*}
Hence, we have $d^\Lambda_t\chi_j =\Lambda (dx_{j}\wedge dx_{n+j}\zeta_{k-1})$. As $\zeta_{k-1}$ is primitive, this just gives $d^\Lambda_t\chi_j=\zeta_{k-1}$. A similar result for $\chi_{n+j}=-x_jdx_{n+j}\wedge \zeta_{k-1}$ also gives $d^\Lambda_t\chi_{n+j}=\zeta_{k-1}$. And adding all these terms together and using (\ref{ajsum}) gives 
\begin{align}
d_t^\Lambda \eta^{\alpha^\prime}=\sum a_j\zeta_{k-1}=-4t b_\beta\zeta_{k-1}. \label{dtLEtaA}\end{align} 
Plugging  the above formula into (\ref{eq:DeltatCformula}), we get 
\begin{align*}
  \left\langle \Delta_{C,t}\begin{pmatrix} \eta^{\alpha^\prime}
\\ b_\beta\zeta_{k-1} \end{pmatrix}, 
\begin{pmatrix} \eta^{\alpha^\prime} \\ b_\beta\zeta_{k-1}\end{pmatrix} \right\rangle&=\lambda_{\eta^\prime}\|\eta^{\alpha^\prime}\|^2-2\langle d_t^{\Lambda}\eta^{\alpha^\prime}, b_\beta\zeta_{k-1}\rangle+\|\Lambda \eta^{\alpha^\prime}\|^2+\|\omega b_\beta\zeta_{k-1}\|^2 \\
  &=\lambda_{\eta^\prime}\|\eta^{\alpha^\prime}\|^2-2\langle -4tb_\beta\zeta_{k-1}, b_\beta\zeta_{k-1}\rangle+\|\Lambda \eta^{\alpha^\prime}\|^2+\|\omega b_\beta\zeta_{k-1}\|^2 \\
&=\lambda_{\eta^\prime}\|\eta^{\alpha^\prime}\|^2+8t\|b_\beta\zeta_{k-1}\|^2+\|\Lambda \eta^{\alpha^\prime}\|^2+\|\omega b_\beta\zeta_{k-1}\|^2  
  \end{align*} 
Thus in this case, the inner product is  bounded by $c\,t$ with  $c=8$. We thus have the right estimate when $\zeta_{k-1}$ is harmonic and the specific $\eta^{\alpha^\prime}$ that belongs to the span of $x_j dx_{n+j}\wedge \zeta_{k-1}$ and $x_{n+j}dx_j\wedge \zeta_{k-1}$ for $j=1,..., n_f$.

Before we continue, we need to make some simple observations on the deformed Laplacian $\Delta_{t}^\Lambda=d_t^\Lambda d_t^{\Lambda*}+d_t^{\Lambda*} d_t^{\Lambda}$.
  
\noindent{\bf (i).} It is useful to express $d_t^\Lambda=d_t\Lambda-\Lambda d_t=*_s d_t *_s$ in terms of the symplectic star operator $*_s\,$, which is defined analogous to the Hodge star operator but with respect to the symplectic structure instead of the Riemannian metric (for a reference, see \cite{TY1}*{Sec. 2.1}).  And similarly, $d_t^{\Lambda*}=\omega d_t^*-d_t^*\omega=*_s d_t^* *_s$.
  Therefore, using that $*_s^2=\mathbb{I}$, we have 
 \begin{align*}
     \Delta^\Lambda_{t}&=d_t^\Lambda d_t^{\Lambda*}+d_t^{\Lambda*} d_t^{\Lambda} \\
     &=*_s d_t *_s*_s d_t^* *_s+*_s d_t^* *_s*_s d_t *_s \\
     &=*_s (d_t d_t^*+ d_t^* d_t) *_s \\
     &=*_s \Delta_{t} *_s
 \end{align*}
  Therefore, since $*_s=*_s^{-1}$, $\Delta^\Lambda_{t}=*_s\Delta_{t}*_s^{-1}$ has the eigenforms $*_s \displaystyle \prod_j H_{m_j}(\sqrt{t}x_j)e^{-t|x|^2/2}dx_{I_k}=\displaystyle \prod_jH_{m_j}(\sqrt{t}x_j)*_sdx_{I_k}$ with eigenvalues $2t(\ell_\Lambda+\sum m_j)$, where $\ell^\Lambda$ is the number of forms in $*_sdx_{I_k}$ missing/different from $dx_{1}\wedge ... \wedge dx_{n_f}$.

\noindent{\bf (ii).} Note that \begin{align*}
      \|d_t^\Lambda \eta_k\|^2&=\langle d_t^\Lambda \eta_k, d_t^\Lambda \eta_k\rangle \\
      &\leq \langle d_t^\Lambda \eta_k, d_t^\Lambda \eta_k\rangle+\langle d_t^{\Lambda*} \eta_k, d_t^{\Lambda*} \eta_k\rangle \\
      &=\langle d_t^{\Lambda*}d_t^\Lambda \eta_k, \eta_k\rangle+\langle d_t^{\Lambda}d_t^{\Lambda*} \eta_k, \eta_k\rangle \\
      &=\langle \Delta_{t}^\Lambda \eta_k, \eta_k\rangle.
  \end{align*}
A similar argument shows  $\|d_t^{\Lambda*}\xi_{k-1}\|^2 \leq \langle \Delta_{t}^\Lambda \xi_{k-1}, \xi_{k-1}\rangle$.

\medskip
  
With these inequalities, we now proceed to show that $$\left\langle \Delta_{C,t}\begin{pmatrix} \eta_k \\ \xi_{k-1}\end{pmatrix}, \begin{pmatrix} \eta_k \\ \xi_{k-1}\end{pmatrix} \right\rangle \geq c\,t\left|\begin{vmatrix}  \eta_k \\ \xi_{k-1}\end{vmatrix}\right|^2+O(\|\eta_k\|^2, \|\xi_{k-1}\|^2)\,.$$ 
%
%
Writing the forms in terms of linear combinations of eigenforms, $\eta_k=\displaystyle \sum_{\alpha}\eta^\alpha$ and $\xi_{k-1}=\sum_\beta\xi^\beta$, as in \eqref{etaexp}-\eqref{xiexp}, the expression for the inner product in   (\ref{eq:DeltatCformula}) becomes
\begin{align}
\label{eq:DeltatCSumformula}
\begin{split}
&    \left\langle \Delta_{C,t}\begin{pmatrix} \displaystyle \sum_{\alpha}\eta^\alpha \\ \displaystyle \sum_{\beta}\xi^\beta \end{pmatrix}, \begin{pmatrix} \displaystyle \sum_{\alpha} \eta^\alpha \\ \displaystyle \sum_{\beta}\xi^\beta\end{pmatrix} \right\rangle \\
    &\qquad=\sum_{\alpha}\lambda_{\eta^\alpha}\langle \eta^\alpha, \eta^\alpha \rangle  -\sum_{\alpha, \beta}2\langle d_t^{\Lambda}\eta^\alpha, \xi^\beta\rangle+ \sum_\beta\lambda_{\xi^\beta}\langle \xi^\beta, \xi^\beta \rangle+\sum_{\alpha, \beta}O(\|\eta^\alpha\|^2, \|\xi^\beta\|^2) 
\end{split}
\end{align}
where $\|\Lambda \eta^\alpha \|^2$ and $\|\om \xi^\beta\|^2$
are bounded by (constant multiple of) the norm of $\|\eta^\alpha\|^2$ and $\|\xi^{\beta}\|^2$. 
To estimate the $\langle d_t^{\Lambda}\eta^\alpha, \xi^\beta\rangle$ term, 
we start by applying the Cauchy-Schwarz inequality and arithmetic-geometric mean inequality to the pair $(\frac{1}{C}\|d_t^\Lambda\eta^\alpha\|^2, C\|\xi^\beta\|^2)$ for any positive constant $C$. Thus
\begin{align}
  \langle d_t^{\Lambda}\eta^\alpha, \xi^\beta \rangle \leq |\langle d_t^{\Lambda}\eta^\alpha, \xi^\beta \rangle| \leq \|d_t^\Lambda \eta^\alpha\|\cdot\|\xi^\beta\|
  \leq \frac{1} {2} (\frac{1} {C}\|d_t^\Lambda\eta^\alpha\|^2+ C\|\xi^\beta\|^2) .\label{argeoIneq}
\end{align} 
Next, note that $\eta^\alpha= a_\alpha \displaystyle\prod_rH_{m_r}(\sqrt{t}x_{r})e^{-t|x|^2/2}dx_{I_k}$ and so we can compute $d_t \eta^\alpha$ using the two standard properties of the Hermite polynomials: 
\[
\frac{\partial} {\partial x_r}H_{m_r}(\sqrt{t}x_r)=\sqrt{t}(2m_r)H_{m_r-1}(\sqrt{t}x_r),\] \[\sqrt{t}x_rH_{m_r}(\sqrt{t}x_r)=\frac{1} {2}H_{m_r+1}(\sqrt{t}x_r)+m_rH_{m_{r}-1}(\sqrt{t}x_r).\] 
Writing $f=f(p) + \sum_j \nu_j\frac{x_j^2}{2}\,$, where $\nu_j=\pm 1$, and so $df =  \sum_j\nu_jx_jdx_j$,  we obtain
\begin{align*}
     d_t \eta^\alpha&=(d+tdf)a_\alpha\prod_rH_{m_r}(\sqrt{t}x_{r})e^{-t|x|^2/2}\wedge dx_{I_k} \\
    &=a_\alpha\left[d(\prod_r H_{m_r}(\sqrt{t}x_{r})e^{-t|x|^2/2})+\sum_j\nu_j tx_jdx_j\prod_r H_{m_r}(\sqrt{t}x_{r})e^{-t|x|^2/2}\right]\wedge dx_{I_k} \\
    &=a_\alpha\left[\sum_{s}\left(\prod_{r\neq s} H_{m_{r}}(\sqrt{t}x_{m_r})\left( \frac{\partial} {\partial x_s}H_{m_s}(\sqrt{t}x_{s})\right)e^{-t|x|^2/2}\right)dx_s\right.\\
    &\quad\left.+\sum_j(\nu_j-1)tx_jH_{m_j}(\sqrt{t}x_{j})e^{-t|x|^2/2}\prod_{r\neq j}H_{m_r}(\sqrt{t}x_r)dx_j)\right]\wedge dx_{I_k} 
\\
    &=a_\alpha\left[\sum_{s}\left(\prod_{r\neq s} H_{m_{r}}(\sqrt{t}x_{m_r})\left(2\sqrt{t} m_sH_{m_s-1}(\sqrt{t}x_{s})e^{-t|x|^2/2}\right)\right)dx_s \right.\\
    &\quad \left.+\sum_j\sqrt{t} (\nu_j-1)\big(\frac{1} {2}H_{m_j+1}(\sqrt{t}x_j)+m_jH_{m_{j}-1}(\sqrt{t}x_j)\big)\prod_{r\neq j}H_{m_r}(\sqrt{t}x_r)dx_j\right] \wedge dx_{I_k}.
\end{align*}
The expression for $d_t \eta^\alpha$ above contains  three distinct terms with $H_{m_{r\neq s}}H_{m_s-1}$ associated with the basic form $dx_s \wedge dx_{I_k}$, and  $H_{m_{r\neq j}}H_{m_j+1}$ and  $H_{m_{r\neq j}}H_{m_j-1}$ both  associated with $dx_j \wedge dx_{I_k}$.  These terms are summed over the $2n$ possible $r$, so we have at most $6n$ terms in $d_t\eta^\alpha$ when $\eta^\alpha$ is basic. Next, note that $\Lambda=\iota_{\partial_{r+n}}\iota_{\partial_{r}}$ can have at most $n$ terms when applied to a basic $\eta^\alpha$. Thus, $d_t^\Lambda \eta^\alpha=d_t\Lambda \eta^\alpha-\Lambda d_t \eta^\alpha$ can have $6n^2$ basic terms for $d_t\Lambda$ and $6n^2$ terms for $\Lambda d_t$ so we get at most $12n^2$ terms (some of these may cancel out). Thus, for each of the basic $\eta^\alpha$, the term $\langle d_t^\Lambda\eta^\alpha, \xi^{\beta}\rangle$ is nonzero for at most $12n^2$ of the $\xi^{\beta}$.  Let us label this set of at most $12n^2$ terms by 
\begin{align}\label{Setadef}
S_{\eta^\alpha}=\{\xi^\beta: \langle d_t^\Lambda \eta^\alpha, \xi^\beta\rangle \neq 0\}.
\end{align}

To bound our inequality, we look at a particular $\alpha$ in \eqref{eq:DeltatCSumformula}, 
\[
\lambda_{\eta^\alpha}\langle \eta^\alpha, \eta^\alpha \rangle  -\sum_{\beta}2\langle d_t^{\Lambda}\eta^\alpha, \xi^\beta\rangle+ \sum_\beta\lambda_{\xi^\beta}\langle \xi^\beta, \xi^\beta \rangle+O(\|\eta^\alpha\|^2, \|\xi^\beta\|^2).
\] 
Ignoring the non $t$ term $O(\|\eta^\alpha\|^2, \|\xi^\beta\|^2)$,  we need to examine the cross terms $-2\langle d_t\eta^\alpha,  \xi^\beta\rangle$. For a particular $\eta^\alpha$, and all the $\xi^\beta$ that have a nonzero contribution (the at most $12n^2$ we found above).   
Thus, if we look at the sum of those terms and use our \eqref{argeoIneq}, we have  
\begin{align*}
\lambda_{\eta^\alpha}\|\eta^\alpha\|^2  -  &\sum_{\xi^{\beta}\in S_{\eta^\alpha}}  2\langle d_t^{\Lambda}\eta^\alpha, \xi^{\beta}\rangle
\geq \lambda_{\eta^\alpha}\|\eta^\alpha\|^2-\sum_{\xi^{\beta}\in S_{\eta^\alpha}} 2|\langle d_t^{\Lambda}\eta^\alpha, \xi^{\beta}\rangle| \\
&\geq    2t(\ell^{\eta^\alpha}+m^{\eta^\alpha})\|\eta^\alpha\|^2-\left(\sum_{\xi^{\beta}\in S_{\eta^\alpha}}\frac{1} {C}\|\Delta_{C,t}^\Lambda\eta^\alpha\|^2 +C\|\xi^{\beta}\|^2\right)\\ 
&\geq 2t(\ell^{\eta^\alpha}+m^{\eta^\alpha})\|\eta^\alpha\|^2-\sum_{\xi_\alpha \in S_{\eta^\alpha}}\frac{2t} {C}(\ell^{\eta^\alpha}_{\Lambda}+ m^{\eta^\alpha})\|\eta^\alpha\|^2 
- C \|\xi^{\beta^\prime}\|^2\\
&=2t\left(\ell^{\eta^\alpha}- \sum_{\xi_\alpha \in S_{\eta^\alpha}} \frac{\ell^{\eta^\alpha}_\Lambda} {C}+\left(1-\sum_{\xi_\alpha \in S_{\eta^\alpha}}\frac{1} {C}\right) m^{\eta^\alpha}\right)\|\eta^\alpha\|^2 
-\sum_{\xi^\beta \in S_{\eta^\alpha}}C\|\xi^{\beta}\|^2.
\end{align*}
Now, recall that $\eta_k$ is not harmonic, so either $\sum m^{\eta^\alpha}\geq 1$ or $\ell^{\eta^\alpha} \geq 1$. Also, $\ell^{\eta^{\alpha}}_\Lambda\leq 2n$, as we can have at most $2n$ terms different from $dx_1 \wedge ... \wedge dx_{n_f}$. Then we can bound the terms above with a positive constant by setting $12^2n^3=C>6n\cdot(12n^2)$, and showing the following is positive
\[
\displaystyle t\left(\ell^{\eta^\alpha}- \sum_{\xi_\alpha \in S_{\eta^\alpha}} \frac{\ell^{\eta^\alpha}_\Lambda} {C}+\left(1-\sum_{\xi_\alpha \in S_{\eta^\alpha}}\frac{1} {C}\right) m^{\eta^\alpha}\right)\|\eta^\alpha\|^2\,.
\]
We note that this is summed over $S_{\eta^\alpha}$, which we know has at most $12n^2$ terms, and changing $\xi^\beta$ does not change this portion of the sum, so \begin{align*}
    2t&\left(\ell^{\eta^\alpha}- \sum_{\xi_\alpha \in S_{\eta^\alpha}} \frac{\ell^{\eta^\alpha}_\Lambda} {C}+\left(1-\sum_{\xi_\alpha \in S_{\eta^\alpha}}\frac{1} {C}\right) m^{\eta^\alpha}\right)\|\eta^\alpha\|^2 \\
    &\geq 2t\left(\ell^{\eta^\alpha}-  \frac{12n^2\ell^{\eta^\alpha}_\Lambda} {C}+\left(1-\frac{12n^2} {C}\right) m^{\eta^\alpha}\right)\|\eta^\alpha\|^2\\  
  &=2t\left(\ell^{\eta^\alpha}-\left(\frac{12n^2\ell^{\eta^\alpha}_\Lambda} {12^2n^3}\right) +\left(1-\frac{12n^2} {12^2n^3}\right)m^{\eta^\alpha}\right)\|\eta^\alpha\|^2 \\
  &= 2t\left(\ell^{\eta^\alpha}-\frac{\ell^{\eta^\alpha}_\Lambda} {6n}+\left(1-\left( \frac{1} {6n}\right)\right)m^{\eta^\alpha}\right)\|\eta^\alpha\|^2 .
\end{align*} 

To investigate the above term, we split into two cases. \\
{\bf Case 1:}  $\ell^{\eta^\alpha} \geq 1$.  Thus
\begin{align*}
2t\left(\ell^{\eta^\alpha}-\frac{\ell^{\eta^\alpha}_\Lambda} {6n}+\left(1-\left( \frac{1} {6n}\right)\right)m^{\eta^\alpha}\right)\|\eta^\alpha\|^2
& \geq 2t\left(\ell^{\eta^\alpha}-\frac{\ell^{\eta^\alpha}_\Lambda} {6n}+\left(1-\left( \frac{1} {6n}\right)\right)0\right)\|\eta^\alpha\|^2 \\
 & \geq 2t(1-\frac{2n} {6n})\|\eta^\alpha\|^2\geq t\|\eta^\alpha\|^2 \,.
 \end{align*}
{\bf Case 2:} Suppose $\ell^{\eta^\alpha}\geq 0$ and $ m^{\eta^\alpha}\geq 1$. Then we have 
\begin{align*}
2t\left(\ell^{\eta^\alpha}-\frac{\ell^{\eta^\alpha}_\Lambda} {6n}+\left(1-\left( \frac{1} {6n}\right)\right)m^{\eta^\alpha}\right)\|\eta^\alpha\|^2
&  \geq 2t\left(0-\frac{2n} {6n}+\left(1-\frac{1} {6}\right) (1)\right)\|\eta^\alpha\|^2 \\
& \geq  2t(\frac{-1} {3}+\frac{5} {6})\|\eta^\alpha\|^2 \geq t\|\eta^\alpha\|^2\,.
 \end{align*}
 Thus, we can conclude that $c^{\eta^\alpha}$ is positively bounded 
 in all cases. So we have 
 \begin{align}
\lambda_{\eta^\alpha}\|\eta^\alpha\|^2-\sum_{\xi^\beta \in S_{\eta^\alpha}} 2\langle d_t^{\Lambda}\eta^\alpha, \xi^{\beta}\rangle\geq t\|\eta^\alpha\|^2-12^2n^3|S_\eta|\|\xi^\beta\|\geq t\|\eta^\alpha\|^2-12^2n^3(12n^2)\|\xi^\beta\|  \label{cEtaIneq}
 \end{align}
 where $S_{\eta^\alpha}$ has at most $12n^2$ elements.
Now looking back at our original expression \eqref{eq:DeltatCSumformula} 
\[ \lambda_{\eta^\alpha}\|\eta^\alpha\|^2  -2\langle d_t^{\Lambda}\eta^\alpha, \xi^{\beta}\rangle+ \lambda_{\xi^\beta}\|\xi^\beta\|^2+O(\|\eta^\alpha\|, \|\xi^\beta\|),\]
if we replace $2\langle d_t^\Lambda \eta^\alpha, \xi^\beta \rangle \leq 2|\langle d_t^\Lambda \eta^\alpha, \xi^\beta\rangle|$ and use our inequality above (which puts the cross terms with a particular $\eta^\alpha$), we then obtain
\begin{align*}
\lambda_{\eta^\alpha}\|\eta^\alpha\|^2  -2\langle d_t^{\Lambda}\eta^\alpha, \xi^{\beta}\rangle&+ \lambda_{\xi^\beta}\|\xi^\beta\|^2+O(\|\eta^\alpha\|, \|\xi^\beta\|)\\ 
&\geq \lambda_{\eta^\alpha}\|\eta^\alpha\|^2  -2|\langle d_t^{\Lambda}\eta^\alpha, \xi^{\beta}\rangle|+ \lambda_{\xi^\beta}\|\xi^\beta\|^2+O(\|\eta^\alpha\|, \|\xi^\beta\|) \\
&\geq t\|\eta^\alpha\|^2-12^2n^3\|\xi^{\beta}\|^2+\lambda_{\xi^\beta}\|\xi^\beta\|+O(\|\eta^\alpha\|,\|\xi^\beta\|) \\
&\geq t\|\eta^\alpha\|^2+\lambda_{\xi^\beta}\|\xi^\beta\|+O(\|\eta^\alpha\|,\|\xi^\beta\|),
\end{align*}  
where we have used the property that $12^2n^3\|\xi^{\beta}\|^2$ is independent of $t$. 

Next let $\eta=\eta^{\alpha^\prime}+\displaystyle \sum_\alpha\eta^\alpha$ and $\xi=\xi^{\beta^\prime}+\displaystyle\sum_{\beta}\xi^\beta$, where $\xi^{\beta^\prime}=b_\beta \zeta_{k-1}$, $\xi^\beta$ are the non-harmonic components, and $\eta^{\alpha^\prime}=\sum (a_{j}x_{n+j}dx_j-a_{n+j}x_jdx_{n+j})\wedge \zeta_{k-1}$ is the component of $\eta$ we know exists to be orthogonal to our two solutions by \eqref{ortoHarm}, and $\eta^\alpha$ are the components orthogonal to $\eta^{\alpha^\prime}$. Note we showed that $\langle d_t^\Lambda \eta^{\alpha^\prime}, \xi_{\beta}\rangle = 0$ for the  $\xi^\beta \neq \xi^{\beta^\prime}$ as these are orthogonal to $d_t^\Lambda \eta^{\alpha^\prime} = b_\beta\zeta_{k-1}$. Using  (\ref{dtLEtaA}), (\ref{cEtaIneq}), and (\ref{eq:DeltatCSumformula}) to both $\eta_k$ and $\xi_{k-1}$ we have 
\begin{align*}
\lambda_{C,t}\left|\begin{vmatrix}\eta_k \\ \xi_{k-1} \end{vmatrix}\right|^2
    &=\left\langle \Delta_{C,t}\begin{pmatrix} \eta^{\alpha^\prime}+\displaystyle\sum_\alpha\eta^\alpha \\ \xi^{\beta^\prime}+\displaystyle\sum_\beta\xi^\beta \end{pmatrix}, \begin{pmatrix} \eta^{\alpha^\prime}+\displaystyle\sum_\alpha\eta^\alpha \\ \eta^{\beta^\prime}+\displaystyle\sum_\beta\xi^\beta \end{pmatrix}\right\rangle \\
    &\geq \lambda_{\eta^{\alpha^\prime}}\|\eta^{\alpha^\prime}\|^2-2\langle d_t^{\Lambda}\eta^{\alpha^\prime}, \xi^{\beta^\prime}\rangle  +0\|\xi^{\beta^\prime}\|^2+ \sum_{\alpha}\lambda_{\eta^\alpha}\|\eta^\alpha\|^2-\sum_{\alpha, \beta, \beta^\prime}2\langle d_t^{\Lambda}\eta^\alpha, \xi^{\beta}\rangle+\\   &\mspace{10mu}\sum_{\beta}\lambda_{\xi^\beta}\|\xi^\beta\|^2+O(\|\eta^\alpha\|, \|\xi^{\beta}\|,\|\eta^{\alpha^\prime}\|, \|\xi^{\beta^\prime}\|) \\
&\geq \lambda_{\eta^{\alpha^\prime}}\|\eta^{\alpha^\prime}\|^2+8t \|\xi^{\beta^\prime}\|^2 +\sum_\alpha\lambda_{\eta^\alpha}\|\eta^\alpha\|^2   -2\sum_{\xi^\beta \in S_{\eta^\alpha}}|\langle d_t^{\Lambda}\eta^{\alpha^\prime}, \xi^{\beta}\rangle|+ \\
&\mspace{10mu}\sum_{\beta}\lambda_{\xi^\beta}\|\xi^\beta\|^2+O(\|\eta^\alpha\|, \|\xi^{\beta}\|,\|\eta^{\alpha^\prime}\|, \|\xi^{\beta^\prime}\|) \\
&\geq t\|\eta^{\alpha^\prime}\|^2+\sum_\alpha t\|\eta^\alpha\|^2+t\|\xi^{\beta^\prime}\|^2-\sum_{\beta} 12^2n^3(12n^2)\|\xi^{\beta}\|^2+\\
&\mspace{10mu}\sum_{\beta}t\|\xi^\beta\|^2+O(\|\eta^\alpha\|, \|\xi^{\beta}\|,\|\eta^{\alpha^\prime}\|, \|\xi^{\beta^\prime}\|) \\
&\geq t\left|\begin{vmatrix}\eta^\alpha \\ \xi^\beta \end{vmatrix}\right|^2-O(\|\eta^\alpha\|,\|\eta^{\alpha^\prime}\|,\|\xi^\beta\|,\|\xi^{\beta^\prime}\|)   
\end{align*} 
where we know each $\lambda_{\xi_\beta}$ is bounded by $2t$. Thus each term in our (possibly infinite) sum has a scale of $t$ while it is  added/subtracted by elements of order $O(\|\eta^\alpha\|,\|\eta^{\alpha^\prime}\|,\|\xi^\beta\|,\|\xi^{\beta^\prime}\|)$ with no factor of $t$. So by driving $t$ large enough it will dominate this inequality (as  has no factor of $t$) and  as eigenforms are non-zero $\lambda>c\,t$ for sufficiently large $t$. 

As another result of our inequality above, we now show that the kernel of $\Delta_{C,t}$ is generated by only our two harmonic solutions. For if it were a third independent solution, we could project it to the orthogonal complement of our two solutions and get a nonzero harmonic solution in the orthogonal complement of $\begin{pmatrix} \zeta_{k} \\ 0 \end{pmatrix}, \begin{pmatrix} -\tau \wedge \zeta_{k-1} \\ \zeta_{k-1} \end{pmatrix}$, but then our inequality above would apply to our new solution, but choosing $t$ large enough would give a positive eigenvalue, contradicting that our projection was harmonic. Thus the kernel of $\Delta_{C,t}$ is two dimensional and is generated by $\begin{pmatrix} \zeta_{k} \\ 0 \end{pmatrix}, \begin{pmatrix} -\tau \wedge \zeta_{k-1} \\ \zeta_{k-1} \end{pmatrix}$.
\end{proof}

\subsection{Local harmonic solutions approximating eigenforms of the deformed cone Laplacian}

In this subsection, we will show that the local harmonic solutions found in Section \ref{harmsol} can approximate global eigenforms of $\Delta_{C,t}$ when $t$ is large.   In particular, as we make clear in Theorem \ref{FCtdim},  the local harmonic solutions represent all the low-lying eigenforms (i.e. those with small eigenvalues) when $t$ is sufficiently large.  The discussion here will parallel that for the Witten-deformed de Rham complex as described in \cite[Section 5.6]{Zhang}.

Without loss of generality, we assume that each $V_p$ where $p \in Crit(f)$ is an open ball of radius $4a$, and assume $t>0$. Let $\gamma_p$ be a smooth bump function such that $\gamma_p(z)=1$ for $|z|\leq a$ and $\gamma_p(z)=0$ for $|z|\geq 2a$. Define 
$$\Phi_{p,1}(t)=\sqrt{\int_{V_p}\gamma_p^2\left\|\begin{matrix} \zeta_{k} \\ \iota_Z\zeta_{k} \end{matrix}\right\|^2_x dV\!ol }\, , \indent \Phi_{p,2}(t)=\sqrt{\int_{V_p} \gamma_p^2 \left\|\begin{matrix} -\tau \wedge \zeta_{k} \\ \zeta_{k}\end{matrix}\right\|_x^2dV\!ol}\, ,$$ 
where $\left\|\begin{matrix} \eta_k \\ \xi_{k-1}\end{matrix} \right\|_x$ is the norm given by $\left\langle \begin{pmatrix} \eta_k \\ \xi_{k-1}\end{pmatrix}, \begin{pmatrix} \eta_k^\prime \\ \xi_{k-1}^\prime\end{pmatrix} \right\rangle_x dVol= \eta_k \w *\eta_k^\prime + \xi_{k-1} \w *\xi_{k-1}^\prime$,  the pointwise form inner product induced by $*$ . We further define 
\begin{equation} 
\rho_{p, 1}(t)=\frac{\gamma_p} {\Phi_{p,1}(t)}\begin{pmatrix} \zeta_{k} \\ \iota_{Z}\zeta_{k}\end{pmatrix},\indent \rho_{p, 2}(t)=\frac{\gamma_p} {\Phi_{p,2}(t)}\begin{pmatrix} -\tau \wedge \zeta_{k} \\ \zeta_{k} \end{pmatrix}\,,\label{rhopi}\end{equation}
which are  unit norm with compact support contained in $V_p$. Note that $\rho_{p,1}(t), \rho_{p,2}(t)$ encompass both generators of local harmonic solutions around $p\in Crit(f)$  described in Corollary \ref{gen2sol}.

Now let $\textbf{H}^1(\tC(\om))$ be the first Sobolev space with respect to a Sobolev norm on $\tC(\om)$. Let $E_{C,t}$ denote the direct sum of the vector spaces generated by the $\rho_{p, 1}(t), \rho_{p, 2}(t)$, and let $E^\perp_{C,t}$ be the orthogonal complement to $E_{C,t}$ in $\textbf{H}^1(\tC(\omega))$, so that $\textbf{H}^1(\tC(\omega))=E_{C,t}\oplus E^\perp_{C,t}$. Let $p_{C,t}, p_{C,t}^\perp$ denote the orthogonal projections from $\textbf{H}^1(\tC(\omega))$ to $E_{C,t}$ and $E_{C,t}^\perp$, respectively, and decompose the operator $D_{C,t}=d_{C, t}+d^*_{C,t}$
via 
$$D_{C,t,1}=p_{C,t}D_{C,t}\,p_{C,t}\,,\  D_{C, t,2}=p_{C,t}D_{C,t}\,p_{C,t}^\perp\,,\ D_{C, t,3}=p_{C,t}^\perp D_{C,t}\,p_{C,t}\,,\ D_{C, t,4}=p_{C,t}^\perp D_{C,t}\,p_{C,t}^\perp\,.$$
We have the following results:
\begin{theorem}
There exists a constant $t_0>0$ such that
\begin{itemize}
\item[(i)] for any $t \geq t_0$, and $0 \leq u \leq 1$, the operator
\begin{align*}
    D_{C, t, u} = D_{C, t,1}+D_{C, t,4}+u(D_{C, t,2}+D_{C, t, 3})=D_{C,t}+(u-1)(D_{C, t,2}+D_{C, t, 3})
\end{align*}
is Fredholm;
\item[(ii)] the operator
$D_{C, t, 4}:E_{C,t}^\perp \cap \textbf{H}^1(\tC(\omega)) \to E_{C,t}^\perp$ is invertible.
\end{itemize}
\end{theorem}
To prove these, we need the following inequalities:
\begin{lemma}\label{Dtineq}
There exists a constant $t_1>0$ such that for $\sigma \in E_{C,t}^\perp \cap \textbf{H}^1(\tC(\omega)), \sigma^\prime \in E_{C,t}$ and $t \geq t_1$, we have
\begin{align*}
\|D_{C,t,2}\sigma\|_0 \leq \frac{C_{1}\|\sigma\|_0} {t}\,,\\
\|D_{C,t,3}\sigma^\prime\|_0 \leq \frac{C_{1}\|\sigma^\prime\|_0} {t}\,,
\end{align*}
for some positive constant $C_{1}$.
\end{lemma}
\begin{proof}
 Note that $D_{C,t,3}$ is the adjoint of $D_{C,t,2}$, so if we prove the first bound, we obtain the second.  Since $\rho_{p,1}(t), \rho_{p,2}(t)$ are supported in $V_p$, we have 
\begin{align*}
    D_{C,t,2}\,\sigma&=\, p_{C,t}\,D_{C,t}\,p_{C,t}^\perp\, \sigma \,=\,p_{C,t}\,D_{C,t}\,\sigma \\
    &=\sum_{p \in Crit(f)} \rho_{p, i}(t) \int_{V_p} \left\langle \rho_{p,i}(t), D_{C,t}\sigma\right\rangle_x dV\!ol \\
    &=\sum_{p \in Crit(f), n_f(p)=k} \rho_{p,1}(t) \int_{V_p} \left\langle D_{C,t}\frac{\gamma_p} {\Phi_{p,1}(t)}\begin{pmatrix} \zeta_{k} \\ \iota_{Z}\zeta_{k}\end{pmatrix}\!,\,\sigma \right\rangle_x dVol \\
    &\qquad +\sum_{q \in Crit(f), n_f(q)=k-1}\rho_{q,2}(t)\int_{V_q}\left\langle D_{C,t}\frac{\gamma_q} {\Phi_{q,2}(t)}\begin{pmatrix} -\tau \wedge \zeta_{k-1} \\ \zeta_{k-1} \end{pmatrix}\!,\, \sigma \right\rangle_x dV\!ol.
\end{align*}
And note that $\gamma_p$ is constant on $|\textbf{x}|<a, |\textbf{x}|>2a$, so $D_{C,t}\rho_{p, 1}(t)=0=D_{C,t}\rho_{p,2}(t)$ (as these are harmonic solutions multiplied by a constant) on $|\textbf{x}|<a, |\textbf{x}|>2a$. \\ 
Now note that we chose 
$$\Phi_{p,1}(t)=\sqrt{\int_{V_p} \gamma_p^2 \left\|\begin{matrix} \zeta_{k} \\ \iota_{Z}\zeta_{k} \end{matrix}\right\|_x^2dV\!ol}\ ,\indent 
\Phi_{p,2}(t)=\sqrt{\int_{V_p} \gamma_p^2 \left\|\begin{matrix} -\tau \wedge \zeta_{k} \\ \zeta_{k}\end{matrix}\right\|_x^2 dV\!ol}\ ,$$ 
so that $\rho_{p,1}(t)=\displaystyle\frac{\gamma_p} {\Phi_{p,1}(t)} \begin{pmatrix} \zeta_{k} \\ \iota_{Z}\zeta_{k} \end{pmatrix}$ and $\rho_{p,2}(t)=\displaystyle\frac{\gamma_p} {\Phi_{p,2}(t)} \begin{pmatrix} -\tau \wedge \zeta_{k} \\ \zeta_{k} \end{pmatrix}$ have unit norm. 
Thus note that, as $\zeta_{k}, 
-\tau \wedge \zeta_{k}, \iota_{Z}\zeta_{k}$ are either $e^{-t|x|^2/2}dx_I$ or $x_je^{-t|x|^2/2}dx_J$ where $dx_J=\iota_{\partial_{n+j}}dx_I$ or $dx_{n+j}\wedge dx_I$. 
Thus, in the region $a<|\textbf{x}|<2a$, these are bounded above by $\int_{\mathbb{R}^n} (1+2a^n)e^{-t|x|^2/2}dV\!ol =\max(1,(2a)^n)\left(\frac{2t} {\pi}\right)^{2n/2}\,,$ 
and bounded below by 
\[\int_{B_a} \max(1,2a^n)e^{-t|x|^2/2}dV\!ol =C_2\left(\frac{2t} {\pi}\right)^{2n/2}\,,\]
and thus $C_3 t^n \leq \frac{1} {\Phi_{p, i}(t)} \leq C_4t^n$. 
We will now look at the integrals 
$$\int_{V_p} \langle D_{C,t}\left(\frac{\gamma_p}{\Phi_{p,1}(t)} \begin{pmatrix} \zeta_{k} \\ \iota_{Z}\zeta_{k} \end{pmatrix}\right), \sigma\rangle dV\!ol\ , \indent
\int_{V_p}\langle D_{C,t}\left(\frac{\gamma_p}{\Phi_{p,2}(t)} \begin{pmatrix} -\tau \wedge \zeta_{k} \\ \zeta_{k} \end{pmatrix}\right), \sigma\rangle dV\!ol\ .$$ 
Note that we can restrict these to $a<|x|<2a$, where $\gamma_p\zeta_{k}, \Lambda\zeta_{k}, \zeta_{k-1}, -\omega \wedge \zeta_{k-1}$ are bounded by $C_5e^{-ta^2}$. Also note that 
\begin{align*}
[d_{t}+d_{t}^*](\gamma_p \psi)&=(d\gamma_p)\wedge \psi+\gamma_p d_t\psi+(d^*+t\iota_\nabla f)\gamma_p\psi \\
&=(d\gamma_p)\wedge \psi+\gamma_p d_t\psi+(\partial_i\gamma_p \psi_I) \iota_{\partial i} dx_I+t\iota_{\nabla f}\gamma_p\psi_I dx_I \\
&=(d\gamma_p)\wedge \psi+\gamma_p d_t\psi+(\partial_i\gamma_p)\psi_I\iota_{\partial_i} dx_I+\gamma_p\partial_i \psi_I \iota_{\partial i} dx_I+\gamma_p t\iota_{\nabla f}\psi_Idx_I \\
&=(d\gamma_p)\wedge \psi+\gamma_p d_t\psi+(\partial_i\gamma_p)\psi_I \iota_{\partial_i}dx_I+\gamma_p d^* \psi_I +t\gamma_p \iota_{\nabla f}\psi \\
&=(d\gamma_p)\wedge \psi+\gamma_p d_t\psi+(\partial_i\gamma_p)\psi_I\iota_{\partial_i}dx_I+\gamma_p d^*_t \psi.
\end{align*} 
Therefore, for $\psi=\zeta_{k}, \iota_Z\zeta_{k}, -\tau \wedge \zeta_{k}, 
$
the terms we can get in $D_{C, t}$ are $\Lambda \psi, \omega\wedge \psi$, which are bounded by $\max(1, 2a)e^{-ta^2/2}$, and $(d\gamma_p)\wedge \psi,\ \gamma_p d_t\psi+\gamma_p d^*_t \psi,\ (\partial_i\gamma_p)\psi_I\iota_{\partial_i}dx_I$ which are bounded by $C_6\max(1, 2a)e^{-ta^2/2}$ for some positive constant $C_6$. Therefore, we have 
\begin{align*}
\int_{V_p} \left\langle D_{C,t} \left(\frac{\gamma_p} {\Phi_{p,1, t}} \begin{pmatrix} \zeta_{k} \\ \iota_{Z}\zeta_{k} \end{pmatrix}\right), \, \sigma\right\rangle_x dV\!ol & \leq \|\sigma\|_0 \sqrt{\int_{a<|\textbf{x}|<2a}\frac{C_7} {\Phi_{p,1,t}} e^{-ta^2/2}dV\!ol } \\
&\leq C_8 t^ne^{-ta^2/2}\|\sigma\|_0 \leq \frac{C_9\|\sigma\|_0} {t}.
\end{align*}
and also 
\begin{align*}
\int_{V_p} \left\langle D_{C,t} \left(\frac{\gamma_p} {\Phi_{p,2, t}} \begin{pmatrix} -\tau \wedge \zeta_{k} \\ \zeta_{k} \end{pmatrix}\right), \, \sigma\right\rangle_x dV\!ol & \leq \|\sigma\|_0 \sqrt{\int_{a<|x|<2a}\frac{C_{10}} {\Phi_{p,2,t}} e^{-ta^2/2}dV\!ol } \\
&\leq C_{11} t^ne^{-ta^2/2}\|\sigma\|_0 \leq \frac{C_{12}\|\sigma\|_0} {t}.
\end{align*}
Therefore, for $C_1=\max(C_{12}, C_8)$ the norm of $D_{C,t,2}\sigma$ satisfies $$\|D_{C,t,2}\sigma\| \leq \left \|\rho_{p,i}(t)\frac{C_{1}\|\sigma\|_0} {t}\right \| =\frac{C_{1}\|\sigma\|_0} {t},$$ and we thus have our first inequality, and thus using properties of adjoints we have proved both inequalities. \end{proof}
Next, note that this implies that $D_{C, t, 2}$ and $D_{C, t, 3}$ are compact operators, and thus $D_{C, t,u}=D_{C,t}+(u-1)(D_{C, t, 2}+D_{C, t, 3})$ is a Fredholm operator plus a compact operator, hence also Fredholm.
To show that the operator $D_{C, t, 4}:E_{C,t}^\perp \cap \textbf{H}^1(\tC(\omega)) \to E_{C,t}^\perp \text{ is invertible }$, we shall follow Bismut and Zhang \cite{BZ} to show that there exists a constant $t_2>0$ such that for $t \geq t_2$ and $\sigma \in E_{C,t}^\perp$, 
\begin{align}
\|D_{C, t,4}\sigma\|_0\geq C_{13}\sqrt{t}\|\sigma\|_0\,. \label{sqrtTineq}  
\end{align} 
To show this, we consider the inequality first in two special cases:
\\[.5cm]
\textbf{Case 1}: $\supp (\sigma) \subset V_p(4a)$ (the ball of radius 4a).
\\
Let $E_p$ be a Euclidean space containing $V_p$, and $\textbf{H}^0(E_p)$ be the Sobolev space corresponding to $\|\cdot\|_0$, the $0$-th Sobolev norm on $E_p$. Define $\rho^\prime_{p,1, t}=\left(\frac{t} {\pi}\right)^{n/2}e^{-t|x|^2/2}\rho_{p,1,t}$ and $\rho^\prime_{p,2, t}=\left(\frac{t} {\pi}\right)^{n/2}e^{-t|x|^2}\rho_{p,t,2}$, and define $p_{C,t}^\prime$ to be the projection onto the subspace of $\textbf{H}^0(E_p)$ spanned by the $\rho_{p,i,t}^\prime$. Since $\sigma \in E_{C,t}^\perp \cap \textbf{H}^1(\tC(\om))$, we have that $p_{C,t}$ projecting $\sigma$ to $E_{C,t}$ is zero, i.e. $p_{C,t}\,\sigma=0$. Accordingly, we have that 
\begin{align*}
p_{C,t}^\prime \,\sigma&=p_{C,t}^\prime\, \sigma -p_{C,t}\,\sigma \\
&= \sum_{p\in Crit(f)} \rho_{p,i,t}^\prime\int_{E_p} (1-\gamma_p(|x|))\left(\frac{t} {\pi}\right)^{n/2}e^{-t|x|^2/2} \langle \rho_{p,i,t}, \sigma\rangle_x dV\!ol.
\end{align*}
As $\gamma_p=1$ near $p$, and zero outside of the ball of radius $4a$, a similar calculation of the integral as was done for the proof of Lemma \ref{Dtineq} shows \begin{align}\|p_{C,t}^\prime\sigma\|^2 \leq \frac{C_{14}} {\sqrt{t}}\|\sigma\|^2\,. \label{recipsqrtTineq} \end{align}
Next, note that $D_{C,t}\rho_{p,i,t}^\prime=0$, so $D_{C,t}p_{C,t}^\prime \sigma =0$, and as $\sigma-p_{C,t}^\prime \sigma \in (E_{C,t}^\prime)^\perp$ we can apply inequalities \ref{sqrtTineq} and \ref{recipsqrtTineq} and get 
$$\|D_{C,t} \sigma\|^2_0=\|D_{C,t}(\sigma-p^\prime_{C,t} \sigma)\|_0^2 \geq C_{13}t\|\sigma-p^\prime \sigma\|_0^2 \geq C_{13}t\|\sigma\|^2_0-C_{16}\sqrt{t}\|\sigma\|^2_0\,. $$ 
Thus, $\|D_{C,t}\sigma\|_0 \geq \frac{C_{17}\sqrt{t}} {2} \|\sigma\|_0\,$. 
\\[.5cm]
\textbf{Case 2}: $\supp (\sigma) \subset M \setminus \displaystyle \bigcup_{p \in Crit(f)} V_p(2a)$ (and still $\sigma \in E_{C,t}^\perp \cap H^1(M)$). \\
To prove the estimate in this case, recall from (\ref{DeltaCt}) that \begin{align*}
D_{C,t}^2&=\Delta_{C,t}=\begin{pmatrix} \Delta_t+\omega \Lambda & -d_t^{\Lambda*} \\ -d_t^\Lambda & \Delta_t+\Lambda\omega  \end{pmatrix} \\
&= \begin{pmatrix} \Delta+t(\mathcal{L}_{\nabla f}+\mathcal{L}^*_{\nabla f}))+t^2||df||_x^2+\omega\Lambda & -d^*+t(\iota_{\nabla f}\omega - \omega \iota_{\nabla f}) \\ -d^\Lambda + t(\Lambda df - df\Lambda) &  \Delta+t(\mathcal{L}_{\nabla f}+\mathcal{L}^*_{\nabla f})+t^2||df||_x^2+\Lambda \omega)\end{pmatrix} \\
&=\begin{pmatrix} \Delta+\omega\Lambda & -d^{\Lambda*} \\ -d^{\Lambda} & \Delta+\Lambda\omega \end{pmatrix}+t\begin{pmatrix} \mathcal{L}_{\nabla f}+\mathcal{L}^*_{\nabla f} & \iota_{\nabla f}\omega -\omega \iota_{\nabla f} \\ \Lambda df-df\Lambda  & \mathcal{L}_{\nabla f}+\mathcal{L}^*_{\nabla f} \end{pmatrix}+t^2||df||_x^2 I. \\
&=\Delta_{C}+t\begin{pmatrix} \mathcal{L}_{\nabla f}+\mathcal{L}^*_{\nabla f} & \iota_{\nabla f}\omega -\omega \iota_{\nabla f} \\ \Lambda df-df\Lambda  & \mathcal{L}_{\nabla f}+\mathcal{L}^*_{\nabla f} \end{pmatrix}+t^2||df||_x^2 I.
\end{align*} 
Since we are away from the zeroes of $df$, $||df||_x^2\geq C_{17}$. Since $\supp(\sigma)$ is away from the zeroes, we have that \[\langle t^2||df||_x^2\sigma, \sigma  \rangle =\int_M t^2||df||_x^2(\eta \wedge *\eta+\xi \wedge * \xi) \geq \int_M t^2C_{17}(\eta \wedge *\eta+\xi \wedge * \xi)=t^2C_{18}\|\sigma\|^2.
\] 
Also, note that $D_{C}^2=d_Cd_C^*+d_C^*d_C$ is a positive operator, hence $\langle D_C\sigma, \sigma\rangle \geq 0\,$. Note further that $\begin{pmatrix} \mathcal{L}_{\nabla f}+\mathcal{L}^*_{\nabla f} & \iota_{\nabla f}\omega -\omega \iota_{\nabla f} \\ \Lambda df-df\Lambda  & \mathcal{L}_{\nabla f}+\mathcal{L}^*_{\nabla f} \end{pmatrix}$ is a zeroth order operator with an operator norm $C_{19}$, so
\begin{align*}
\|D_{C,t}\sigma\|^2 &= \langle D_{C,t}\sigma, D_{C,t}\sigma\rangle =\langle D^2_{C,t}\sigma,\sigma\rangle \\
&=\left\langle \left(D_C^2+t\begin{pmatrix} \mathcal{L}_{\nabla f}+\mathcal{L}^*_{\nabla f} & \iota_{\nabla f}\omega -\omega \iota_{\nabla f} \\ \Lambda df-df\Lambda  & \mathcal{L}_{\nabla f}+\mathcal{L}^*_{\nabla f} \end{pmatrix}+t^2|df|^2I\right)\sigma,\sigma\right\rangle\\
&=\left\langle \Delta_C \sigma, \sigma \right\rangle+t\left\langle\begin{pmatrix} \mathcal{L}_{\nabla f}+\mathcal{L}^*_{\nabla f} & \iota_{\nabla f}\omega -\omega \iota_{\nabla f} \\ \Lambda df-df\Lambda  & \mathcal{L}_{\nabla f}+\mathcal{L}^*_{\nabla f} \end{pmatrix}\sigma, \sigma\right\rangle+t^2\langle |df|^2\sigma,\sigma\rangle \\
& \geq (0-C_{19}t+C_{18}t^2)\|\sigma\|^2,
\end{align*} 
from which we can conclude $\|D_{C,t}\sigma\| \geq C_{20}\sqrt{t} \|\sigma\|\,.$

With the two cases at hand, we now derive the inequality (\ref{sqrtTineq}). First, we define the function $\tilde{\gamma_p} \in C^\infty(M)$ such that $\tilde{\gamma_p}(y)=\gamma_p(|y|/2)$ in $V_p$ and $\tilde{\gamma_p}|_{M\setminus \bigcup V_p(4a)}=0$.  For $\sigma \in E_{C,t}^\perp \cap \textbf{H}^1(M)$, one can see that $\tilde{\gamma_p}\sigma \in E_{C,t}^\perp \cap \textbf{H}^1(M)$.
Also, $\|D_{C,t}\sigma\| \geq \|D_{C,t} \sigma - \tilde{\gamma_p}D_{C,t}\sigma\|+\|\tilde{\gamma_p}D_{C,t}\sigma\|\,$.

Therefore, with Case 1 and 2, we deduce that there is a $C_{21}$ such that for $t \geq t_1+t_2$, 
\begin{align*}
\|D_{C,t}\sigma\|_0 &\geq \frac{1} {2}(\|(1-\tilde{\gamma_p})D_{C,t}\sigma\|_0+\|\tilde{\gamma_p}D_{C,t}\sigma\|_0) \\
&\geq \frac{1} {2}\left(\|D_{C,t}(1-\tilde{\gamma_p})\sigma+[D_{C,t}, \gamma_p]\sigma\|_0+\|D_{C,t}\tilde{\gamma_p}\sigma+[\gamma_p, D_{C,t}]\sigma\|_0\right) \\
&\geq \frac{\sqrt{t}} {2} (C_{20}\|(1-\tilde{\gamma_p})\sigma\|_0+\sqrt{C_{17}}\|\tilde{\gamma_p}\sigma\|_0)-C_{21}\|\sigma\|_0 \\
&\geq \sqrt{t}C_{22} \|\sigma\|_0-C_{21}\|\sigma\|_0
\end{align*}
where $C_{22}=\min\{\sqrt{C_{17}}/2, C_{20}/2\}$. Thus, we have completed the proof of inequality (\ref{sqrtTineq}) from which we conclude that the operator $D_{C, t,4}:E_{C,t}^\perp \cap \textbf{H}^1(M) \to E_{C,t}^\perp$ is invertible when $t$ is sufficiently large.
\begin{lemma}\label{Dct1}
There exists a constant $C_{23}>0$ such that for $\sigma \in {\bf H}^1(\tC(\omega))$ and any $t\geq t_3$, we have
$$\|D_{C, t,1} \sigma\|_0 \leq \frac{C_{23}\|\sigma\|_0} {t}\,.$$
\end{lemma}
\begin{proof} 
If we examine the integral, then using the $\rho_{p,i}$'s in \eqref{rhopi} as our basis for $E_{C,t}$, we find 
\begin{align*}
p_{C,t}\,\sigma&=\sum_{p \in Crit(f)} \rho_{p, i}(t) \int_{V_p} \langle \rho_{p,i}(t), \sigma\rangle_x dV\!ol\,.
\end{align*}
And note that if we take $D_{C,t}\,  \rho_{p, i}(t)$, then this is zero in the region $|x-p|<a$ and $|x-p|>2a\,$, and from a similar argument to Lemma \ref{Dtineq}, $\|D_{C,t}\,  \rho_{p, i}(t)\| \leq \frac{C_{12}} {t}$. Then using that $ \rho_{p, i}(t)$ has unit norm, we obtain
\begin{align*}
\|p_{C,t}D_{C,t}\,p_{C,t} \sigma\|_0&=\left \| \sum_{p \in Crit(f)} \rho_{p, i}(t) \int_{V_p} D_{C,t}\,\rho_{p,i}(t) \langle \rho_{p,i}(t), \sigma\rangle_x dV\!ol \right \|_0 \\
&\leq \sum_{p \in Crit(f)} \left\| \,\rho_{p, i}(t) \,\langle_x \rho_{p,i}(t), \sigma\rangle \int_{V_p} D_{C,t}\,\rho_{p,i}(t)dV\!ol  \right\| \\
    &\leq \sum_{p \in Crit(f)} \left\|\langle \rho_{p,i}(t), \sigma\rangle \frac{C_{12}} {t}  \right\| \\
    &\leq \sum_{p \in Crit(f)}\frac{C_{12}} {t}  \|\rho_{p,i}(t)\|_0 \,\|\sigma\|_0 \\
     &\leq\frac{C_{23}\|\sigma\|_0} {t}.  
\end{align*}
\end{proof}
\begin{definition}
For any $b>0$, let $E_{C,t}(b)$ denote the direct sum of eigenspaces of $D_{C, t}$ with eigenvalues in $[-b, b]$. Since $D_{C, t}$ is a self-adjoint linear operator, $E_{C,t}(b)$ is a finite-dimensional subspace of $\textbf{H}^0(\tC(\omega))\,.$
\end{definition}
Let $p_{C,t}(b)$ denote the projection operator from $\textbf{H}^0(\tC(\omega))$ to $E_{C,t}(b)\,$.
\begin{lemma} \label{PCtineq}
There exists a $C_{24} > 0$ such that for $t \geq t_4$ and $\sigma \in E_{C,t}\,,$ 
$$\|p_{C,t}(b)\sigma-\sigma\|_0 \leq \frac{C_{24}} {t} \|\sigma\|_0.
$$
\end{lemma}
\begin{proof} Let $\delta=\{\lambda \in \mathbb{C}: |\lambda|=C_\lambda\}$ be the counterclockwise oriented circle at radius $C_\lambda$. By Lemma \ref{Dtineq} and Lemma \ref{Dct1}, we have that for any $\lambda \in \delta, t\geq t_1+t_2\,$, and $\sigma^\prime \in \textbf{H}^1(\tC(\omega)$,
\begin{align*}
\|(\lambda-&D_{C, t})\sigma^\prime\|_0\\
&\geq \frac{1} {2}\left(\|\lambda p_{C,t} \sigma^\prime -D_{t,1}p_{C,t}\sigma^\prime-D_{t,2}p_{C,t}^\perp\sigma^\prime\|_0+\|\lambda p_T^\perp\sigma^\prime-D_{t,3}p_{C,t}\sigma^\prime-
D_{t,4}p_{C,t}^\perp\sigma^\prime \|_0\right)\\
&\geq \frac{1} {2}\left(\left(C_\lambda-\frac{C_{12}} {t} -\frac{C_{23}} {t}\right)\| p_{C,t}\sigma^\prime\|_0+\left(C_{17}\sqrt{t}-C_\lambda-\frac{C_{23}} {t}\right) \|p_{C,t}^\perp\sigma^\prime\|_0\right).
\end{align*}
By the above inequality, for $t_4 \geq t_1+t_2$ and $C_{25}>0$ such that for any $t>t_4$ and $\sigma^\prime \in \textbf{H}^1(\tC(\omega))$ \\
$$\|(\lambda-D_{C,t})\sigma^\prime\|_0 \geq C_{25}\|\sigma^\prime\|_0\,.$$ 
Thus, for any $\lambda \in \delta$, $\lambda-D_{C,t}:\textbf{H}^1(\tC(\omega)) \to \textbf{H}^0(\tC(\omega))$ is invertible, so the resolvent $(\lambda-D_{C,t})^{-1}$ is well defined. By the basic spectral theory for operators, for $\sigma \in E_{C,t}$, one has
\begin{align*}
p_{C,t}(b)\sigma-\sigma &= \frac{1}{2\pi\sqrt{-1}}\int_{\delta}((\lambda-D_{C,t})^{-1}-\lambda^{-1})\sigma d\lambda\,.
\end{align*}
Since with $p_{C,t}$ the projection to $E_{C,t}$, we have $p_{C,t}^\perp \sigma =0\,$. Thus, using the inequality above, we have 
\begin{align*}
   ( (\lambda-D_{C,t})^{-1}-\lambda^{-1})\sigma&=\lambda^{-1}(\lambda-D_{C,t})^{-1}D_{C,t}\sigma\\
    &=\lambda^{-1}(\lambda-D_{C,t})^{-1}(D_{C, t,1}\sigma+D_{C, t,3}\sigma).
\end{align*}
One deduces by Lemma \ref{Dtineq} and above we have
\begin{align*}
    \|(\lambda-D_{C,t})^{-1}(D_{C, t,1}\sigma+D_{C, t,3}\sigma)\|_0&\leq C^{-1}_{13}\|D_{t,1}\sigma+D_{t,3}\sigma\|_0\\
    &\leq C^{-1}_{25}\left(\frac{C_{12}+C_{23}} {t}\right) \|\sigma\|_0
\end{align*}
and plugging this into the integral gives 
\begin{align*}
    \|p_{C,t}(b)\sigma-\sigma \|_0&= \left \|\frac{1}{2\pi\sqrt{-1}}\int_{\delta}((\lambda-D_{C,t})^{-1}-\lambda^{-1})\sigma d\lambda \right \|_0\\
    &\leq \frac{1}{2\pi} \int_\delta \|\lambda^{-1}(\lambda-D_{C,t})^{-1}(D_{C, t,1}\sigma+D_{C, t,3}\sigma) \|_0d\lambda \\
    &\leq \frac{C_{26}}{2\pi} \int_\delta \|C^{-1}_{25}\left(\frac{C_{12}+C_{23}} {t} \right)\|\sigma\|_0d\lambda \\
    &\leq \frac{C_{24}} {t} \|\sigma\|_0.
\end{align*} 
\end{proof}
\begin{theorem}  \label{FCtdim}
Let $F_{C, t}^{[0,b]}$ be the space of all eigenforms of $\Delta_{C,t}$ with eigenvalues in $[0,b]$. Then for $t$ large enough, $(F_{C, t}^{[0,b]}, d_{C,t})$ is a chain complex with $\dim\,(F^{[0,b]}_{C,t})_k=m_k+m_{k-1}$.
\end{theorem}
\begin{proof}
By applying Lemma \ref{PCtineq} to the $\rho_{p, i, t}$'s when $t$ is large enough, $p_{C, t}(b)\rho_{p, i, t}$ will be linearly independent. (If they are not linearly dependent, then $\sigma=\rho_{p, i, t}$ and $\sigma^\prime=a\rho^\prime_{p^\prime, i^\prime, t}$ would have $p_{C,t}(b)\sigma=p_{C,t}(b)\sigma^\prime$, but then by Lemma \ref{PCtineq}, we would have $$\frac{C_1}{t}\|\sigma-\sigma^\prime\|_0 \geq \|p_{C,t}(b)\sigma-p_{C,t}(b)\sigma^\prime-(\sigma-\sigma^\prime)\|_0=\|\sigma-\sigma^\prime\|_0$$  which is a contradiction.)
Thus for $t\geq t_5$, we have $\dim(E_{C, t}(b))\geq \dim(E_{C,t})$. Now assume for the purposes of contradiction that $\dim(E_{C, t}(b))\geq \dim(E_{C,t})$. Then there is a nonzero $\sigma \in E_{C, t}(b)$ that is orthogonal to $p_{C,t}(b)E_{C,t}$, or $\langle \sigma, p_{C,t}\rho_{p,i}(t)\rangle_{\textbf{H}^0(\tC(\omega)}=0$ for any $\rho_{p,i}$. Then from Lemma \ref{Dtineq} and Case 1, we have that
\begin{align*}
    p_{C,t}\,\sigma&=\sum_{p \in Crit(f)} \langle \sigma, \rho_{p,i}(t)\rangle \rho_{p,i}(t) \\
    &=\sum_{p \in Crit(f)} \langle \sigma, \rho_{p,i}(t)\rangle \rho_{p,i}(t)-\sum_{p \in Crit(f)} \langle \sigma, p_{C, t}\,\rho_{p,i}(t)\rangle p_{C,t}(b)\,\rho_{p,i}(t) \\
    &=\sum_{p \in Crit(f)}\!\!\!\! \langle \sigma, \rho_{p,i}(t)\rangle (\rho_{p,i}(t)-p_{C,t}(b))\,\rho_{p,i}(t))+\!\!\!\!\!\!
    \sum_{p \in Crit(f)} \!\!\!\!\!\langle \sigma, \rho_{p,i}(t)-p_{C, t}\,\rho_{p,i}(t)\rangle p_{C,t}(b)\,\rho_{p,i}(t).
\end{align*}
By Lemma \ref{Dtineq}, there exists a $C_{12}>0$ so when $t\geq t_5$ $\|p_{C,t}\sigma\|_0 \leq \frac{C_{12}} {t} \|\sigma\|_0$, and thus $$\|p_{C,t}^\perp \sigma \|_0= (\|\sigma\|_0-\|p_{C,t}\sigma\|_0) \geq \|\sigma\|_0\geq C_{25}\|\sigma\|_0\,. $$
Using this, (\ref{sqrtTineq}), Lemma \ref{PCtineq}, and when $t>0$ is large enough, we have 
\begin{align*}
    C_{25}C_{20}\sqrt{t}\|\sigma\|_0 &\leq C_{20}\sqrt{t}||p_{C,t}^\perp \sigma||_0\\ &\leq \|D_{C, t}p_{C,t}^\perp\sigma\|_0 \\
    &=\|D_{C, t}\sigma-D_{C,t}p_{C,t}\sigma\|_0\\
    &=\|D_{C, t}\sigma-D_{C, t,1}\sigma-D_{C, t,3}\sigma\|_0 \\
    &\leq \|D_{C, t}\sigma\|_0+\|D_{C, t,1}\sigma\|_0+\|D_{C, t,3}\sigma\|_0 \\
     &\leq \|D_{C, t}\sigma\|_0+\frac{C_{12}+C_{23}} {t}\|\sigma\|_0,
\end{align*}
 from which one gets $\|D_{C,t}\sigma\|_0 \geq C_{25}C_{20}\sqrt{t}\|\sigma\|_0-\frac{C_{12}+C_3} {t}\|\sigma\|_0$ which contradicts that $\sigma \in E_{C,t}(b)$ is an eigenspace of $D_{C,t}$ for $t$ large enough. Thus, one has 
\[\dim(E_{C,t}(b))=\dim E_{C,t}=\sum_{k} m_{k}+m_{k-1}=2\sum_{k} m_k\,. \] 
Moreover, $E_{C,t}$ is generated by $p_{C,t}(b)\,\rho_{p,i}(t)$. 
\end{proof}

Now to prove Theorem \ref{FCtdim}, for any integer $k$, such that $0\leq k \leq 2n+1$, let $Q_i$ denote the projection from $\textbf{H}^0(\tC(\omega))$ onto the $L^2$ completion of $\tC^k(\omega)$. Since $\Delta_{C, t}$ preserves the $\mathbb{Z}$ grading of $\Omega^*(M)$, for any eigenvector $\sigma$ of $D_{C,t}$ associated with an eigenvalue $\mu \in [-b,b]$ 
\[\Delta_{C,t}Q_k\sigma=Q_k\Delta_{C,t}\sigma=Q_k\mu^2\sigma=\mu^2Q_k\sigma\,.\] That is, $Q_k\sigma$ is an eigenform of $\Delta_{C,t}$ with eigenvalue $\mu^2$.  We thus need to show that $\dim Q_kE_{C,t}(b)=m_k+m_{k-1}$. To prove this, note that by Lemma \ref{PCtineq},
\[\|Q_{n_f(p)}p_{C,t}(b)\rho_{p,i}(t)-\rho_{p,i}(t)\|_0\leq \frac{C_{24}} {t}\,.\]
Thus, for $t$ sufficiently large, the cone forms $Q_{n_f(p)}p_{C,t}(b)\,\rho_{p,i}(t)$ are linearly independent. Therefore, for each $k$, we have
\[\dim Q_kE_{C,t}(b)\geq m_k+m_{k-1}\,. \]
However, we also have (as every element in $\textbf{H}^0(\tC(\omega))$ is a linear combination of $2n+1$ form) 
\[\sum_{k=0}^{2n+1} \dim Q_kE_{C,t}(b)=\sum_{k=0}^{2n+1} \dim E_{C,t}(b) = \sum_{k} m_{k}+m_{k-1}=2\sum_{k} m_k\,. \]
From this and $\dim Q_kE_{C,t}(b)\geq m_k+m_{k-1}$, we obtain 
$$\dim Q_kE_{C,t}(b)=m_k+m_{k-1}\,.$$


\section{Relation between the cone complex and the cone Morse complex}

For $(F_{C, t}^{[0,1]})_k$, the space of all eigenforms of $\Delta_{C,t}$ in $ \tC^k(\om)$ with eigenvalues in $[0,1]\,$, we point out that $d_{C,t}: (F^{[0,1]}_{C,t})_k \to (F^{[0,1]}_{C,t})_{k+1}$, since $[\Delta_{C,t}, d_{C,t}]=0\,$.  Hence,  $((F_{C, t}^{[0,1]})_\bullet, d_{C,t})$ is a cochain complex with cohomology, $H^k(F^{[0,1]}_{C, t})\cong H^k(\tC(\om))\,$. By Theorem \ref{FCtdim}, when $t$ is sufficiently large, we have
\begin{align}\label{Fctbound}
\dim\,(F^{[0,1]}_{C, t})_k=m_k+m_{k-1}\,,\qquad k=0,1, \ldots, 2n+1\,.
\end{align}
which is precisely equal to the dimension of the cone Morse complex $\tC^k(\com)=C^k(M,f) \oplus C^{k-1}(M, f)$ defined in Definition \ref{delCdef}.  As an immediate corollary, we have the bound
\begin{align*}
    b_k^\om=\dim H^k(\tC(\om))\leq m_k + m_{k-1}\,,\qquad k=0,1, \ldots, 2n+1\,. 
\end{align*}
However, this inequality does not give a sharp bound. 
Below, we shall proceed to prove the isomorphism of $H^k(\tC(\om))\cong H^k(\tC(\com))$ Theorem \ref{MIso}, which will allow us to derive the sharp Morse inequality bounds described in Theorem \ref{CMineq}.

\subsection{Quasi-isomorphism between $\tC(\om)$ and $\tC(\com)$}

In this subsection, we will prove Theorem \ref{MIso} by showing $H^k(\tC(\om))\cong H^k(\tC(\com))$. To do so, let us first briefly review the relationship between de Rham cohomology, $H^k_{dR}(M)$, and Morse cohomology, $H^k_{C(f)}(M)$.
Recall that there is a map $\mP$ that links the de Rham complex with the Morse complex \cite{BZ}.
\begin{definition}\label{mP}
  Define the map $\mP:\Omega^k(M) \to C^k(M,f)$ by   
  \[ \mP\phi= \sum_{p_k \in Crit(f)} \left( \int_{\overline{U}_{p_k}} \phi\right)p_k \]
where $\phi\in \Om^k(M)$ and  $U_p$ is the unstable submanifold consisting of gradient flow lines moving away from $p$.
\end{definition}
Importantly, $\mP$ is a chain map and induces an isomorphism on cohomology. 
\begin{theorem}
(\cite[Theorem 2.9]{BZ})\label{thm:mapP}
The map $\mP:\Omega^k(M) \to C^k(M,f)$ is a chain map, i.e. 
\begin{align}\label{Pchain}
\del\, \mP=\mP\, d\,,
\end{align}
and moreover, 
\begin{align}\label{Pisom}
[\mP]:H^k_{dR}(M) \to H^k_{C(f)}(M) \text{ is an isomorphism.} 
\end{align}
\end{theorem}
The analytical Witten deformation proof of this theorem can be described by the following diagram:
\begin{equation}\label{DZhang}
\begin{tikzcd} 
(\Omega^\bullet(M), d)  \arrow[d,swap,"\mP"] & (\Omega^\bullet(M), d_{t}) \arrow[l, swap, "e^{tf}"] \arrow[ld,"\mP_{t}"] & ((F_{t}^{[0,1]})_\bullet, d_{t}) \arrow[l,swap, "\iota_{t}^{[0,1]}"]\arrow[lld,"\mP_{t}\big|_{F_t^{[0,1]}}
", shift left=.16cm] \\[2ex] 
(C^\bullet(M,f), \partial) &  & 
\end{tikzcd}
\end{equation}
where $(\Omega^\bullet(M), d_{t})$ is the Witten-deformed de Rham complex, and similar to $((F_{C, t}^{[0,1]})_\bullet, d_{C,t})$, $((F_{t}^{[0,1]})_\bullet, d_{t})$ is the cochain complex consisting of eigenforms of $\Delta_t$ in $\Om(M)$ with eigenvalues in $[0,1]$.  In the diagram, $\iota_t^{[0,1]}$ is the inclusion map, and also, the $\mP$ map induces the map 
\begin{align}
\mP_t=\mP \,e^{tf}: (\Om^k(M), d_t) \to (C^k(M,f),\partial)\,.\end{align} 
The proof of Theorem \ref{thm:mapP} involves showing that for $t$ large enough, $\mP_t\big|_{F_t^{[0,1]}}:((F_t^{[0,1]})_\bullet, d_t) \to (C^\bullet(M,f), \partial)$ is a cochain isomorphism \cite{Zhang}. This implies that the vertical map $\mP$ gives an isomorphism on cohomology since the cohomologies of $\Omega^\bullet(M)$ and $F_t^{[0,1]}$ are always identical, regardless of the value of $t$.
 
In considering $\tC(\com)$, let us first recall the definition of the map $c(\om):C^k(M,f) \to C^{k+2}(M,f)$ when acting on a critical point $p\in Crit(f)$ with index $n_f(p)=k$:
\begin{align}\label{commap}\displaystyle \com\, p = \sum_{q\in Crit(f)} \left(\int_{\overline{\CM(q,p)}} \om\right)\,q\:,
\end{align} 
where the sum is over critical points $q$ with index $n_f(q)=n_f(p)+2$ and ${\CM(q,p)}$ is the submanifold of flow lines from $q$ to $p$. It was shown in \cite[Section  3.5]{AB} and \cite[Lemma 4]{Viterbo}
that 
\begin{align}\label{Pomcoh}
[\mP] [\om]=[\com][\mP]\,,  
\end{align} 
that is, they are cohomologous as maps from $H^k_{dR}(M)$ to $H^{k+2}_{C(f)}(M)$.

In the following, we assume that $t$ is sufficiently large such that $\mP_t = \mP e^{tf}$ is an isomorphism between $(F_t^{[0,1]})_k$ and $C^k(M, f)$. 
By \eqref{Pomcoh}, we have that 
\begin{align*}
[\mP_t] [\om] = [\com][\mP_t]  
\end{align*}
on cohomology.
This motivates us to introduce an $\om \w$ type map on $F_t^{[0,1]}$.
\begin{definition}\label{omtilde} 
For $t$ sufficiently large, we define the map $\tilde{\omega}_t: (F_t^{[0,1]})_k \to (F_t^{[0,1]})_{k+2}$  by  \[\tilde{\omega}_t=\mP_t^{-1}\com \mP_t. \]
\end{definition} 

Since $d_t = e^{-tf} d\, e^{tf}$, we can extend $\mP\,d = \partial\, \mP$ of  \eqref{Pchain} to  
\begin{align}\label{Ptdt}
\mP_t\,d_t=\partial\, \mP_t\,.   
\end{align}  
Applying $\mP_t^{-1}$ on both the left and the right of \eqref{Ptdt} gives  $d_t\mP_t^{-1}=\mP_t^{-1}\partial$ acting on $C(M,f)$. 
Using the commutativity of $\com$ with $\partial$, i.e. $\partial \com = \com \partial$ in \eqref{MLeibniz}, it is straightforward to check that $\tilde{\omega}$ is a chain map:
\begin{align*}
 \tilde{\omega}_td_t&=\mP_t^{-1}\com \mP_td_t \\
&=\mP_t^{-1}\com \partial \mP_t \\
&=\mP_t^{-1}\partial\com\mP_t \\
&=d_t\,\mP_t^{-1}\com \mP_t = d_t\tilde{\omega}_t.
\end{align*}
The induced map on the  cohomology
\begin{align}\label{wtcr}
[\tilde{\omega}_t]=[\mP_t^{-1}][\com][\mP_t]=[\mP_t^{-1}][\mP][\omega \wedge][\mP]^{-1}[\mP_t]\,,
\end{align}
is thus conjugate to the wedge product map $\omega \wedge $. With $\tilde{\omega}_t$, we can use it to define the following cone complex: 
\begin{align}\label{wtcmplx}
\tC^k(\tilde{\omega}_t)=(F_{t}^{[0,1]})_k\oplus (F_{t}^{[0,1]})_{k-1}\,, \qquad\tilde{d}_{C,t}=\begin{pmatrix}d_t & \tilde{\omega}_t \\ 0 & -d_t \end{pmatrix}\,.
\end{align}
The cohomology of this cone Morse complex $(\tC^\bullet(\tom), \td_{t})$ can be expressed in terms of the cokernels and kernels of the $\tom$ map in the following way:
\begin{align}\label{wtckk}
H^k(\tC(\tom))(M) &\cong \coker\left[[\tom]:H^{k-2}(F_t^{[0,1]}) \to H^{k}(F_t^{[0,1]})\right]\nonumber \\ &\qquad \oplus \ker\left[[\tom]: H^{k-1}(F_t^{[0,1]}) \to H^{k+1}(F_t^{[0,1]})\right]\,.
\end{align}  
This relation follows from the fact that the cone complex with elements $\tC^k(\tom)$ sits in a short exact sequence of chain complexes
\begin{equation}\label{wtses}
\begin{tikzcd}
0  \arrow[r] & ((F_t^{[0,1]})_k, \td_t) \arrow[r, "\iota"] & (\tC^k(\tom), \td_{C,t}) \arrow[r, "\pi"] & ((F_t^{[0,1]})_{k-1}, -\td_t) \arrow[r]& 0 
\end{tikzcd}
\end{equation}
where $\iota$ is the inclusion map and $\pi$ is the projection onto the second component.  The short exact sequence implies the following long exact sequence of cohomologies
\begin{equation}\label{wtHes}
\begin{tikzcd} 
\ldots \arrow[r, ] &
 H^{k-2}(F_t^{[0,1]}) \mspace{-1mu}\arrow[r, "\text{[}\tom\text{]}"] & H^{k}(F_t^{[0,1]}) \arrow[draw=none]{d}[name=X, anchor=center]{} \arrow[r, "\text{[}\iota\text{]}"] & H^{k}(\text{Cone}(\tom)) \ar[rounded corners,
            to path={ -- ([xshift=2ex]\tikztostart.east)
                      |- (X.center) \tikztonodes
                      -| ([xshift=-2ex]\tikztotarget.west)
                      -- (\tikztotarget)}]{dll}[at end, swap]{ \text{[}\pi\text{]~~~~~~~~~~~~~~~~~~~~~~~~~~~~~~~~~~~~~~~~~~~~~~~~~~~~~~~~~~~~~~~~~~~~~~~~~~~~~~~~~~~~~~~~~~~~~}}& \left. \right. \\ & H^{k-1}(F_t^{[0,1]}) 
\arrow[r, "\text{[}\tom\text{]}"] & H^{k+1}(F_t^{[0,1]})\arrow[r, ] & \ldots \text{~~~~~~~~~~~~~~}
\end{tikzcd}
\end{equation}
which implies \eqref{wtckk}.

We now point out two important properties of $H^k(\tC(\tom))$ when $t$ is sufficiently large such that $\mP_t$ is an isomorphism. 
\begin{itemize}
\item[(i)] $H^k(\tC(\tom))\cong H^k(\tC(\com))$.  By construction, for $t$ sufficiently large, the complex $(\tC^\bullet(\tom), \td_{C,t})$ is isomorphic to the cone Morse complex $(\tC^\bullet(\com), \del_C)$. Hence, their cohomologies must be isomorphic.
\item[(ii)]$H^k(\tC(\tom))\cong H^k(\tC(\om))$.  Being both cone cohomologies, both $H^k(\tC(\tom))$ and $H^k(\tC(\om))$ can be expressed in terms of the cokernels and kernels, of the $[\tom]$ map \eqref{omckk} and the $[\om]$ map \eqref{wtckk}, respectively.  Moreover, since $H^k(F_t^{[0.1]})\cong H^k_{dR}(M)$ and also $[\tom]$ and $[\om]$ have equivalent action on the cohomology level by \eqref{wtcr}, the two cohomologies are isomorphic. 
\end{itemize}
Together, they imply the desired isomorphism that $ H^k(\tC(\com))\cong H^k(\tC(\om))\cong PH^k(M,\om)$ for $k=0, 1, \ldots, 2n+1\,$.  And this proves Theorem \ref{MIso}.

\subsection{Cone Morse inequalities}

Having proved $H(\tC(\om))\cong H(\tC(\com))$, we now proceed to derive Morse-type bounds for $b_k^\om = \dim PH^k(M,\om)= \dim H^k(\tC(\om))$.

To do so, we note that the cohomology of the cone Morse complex $(\tC^\bullet(\com), \partial_C)$, like any cone cohomology, can be expressed in terms of cokernels and kernels of the $\com$ map:
\begin{align}\label{MCckk}
H^k(\tC(\com))(M) &\cong \coker\left[[\com]:H_{C(f)}^{k-2}(M) \to H_{C(f)}^{k}(M)\right]\nonumber \\ &\qquad \oplus \ker\left[[\com]: H_{C(f)}^{k-1}(M) \to H_{C(f)}^{k+1}(M)\right]\,
\end{align}
where $H_{C(f)}(M)$ is the cohomology of the standard Morse cochain complex $(C^\bullet(M,f), \partial)$.  This relation can be derived similarly as that for $H^k(\tC(\tom))$ in \eqref{wtckk} by means of short exact sequence of chain complexes \eqref{wtses} resulting in a long exact sequence of cohomologies \eqref{wtHes}.

Since the dimensions of the cohomology of the Morse complex are given by the Betti numbers, i.e. $b_k =\dim H^k_{C(f)}(M)$, it follows from the isomorphism $H^k(\tC(\om))\cong H^k(\tC(\com))$ 
and \eqref{MCckk} that
\begin{align}\label{bok0}
b_k^\om = \dim H^k(\tC(\om))= b_k - r_{k-2} + b_{k-1}-r_{k-1}
\end{align}
where 
\begin{align}  \label{rkdef}  
r_k &= \rk \left( [\com]: H^{k}_{C(f)}(M) \to  H^{k+2}_{C(f)}(M)\right) \\
&=\rk\left( [\om]: H^{k}_{dR}(M) \to  H^{k+2}_{dR}(M)\right) \nonumber
\end{align} 
is the rank of the $\com$ map on $H^k_{C(f)}(M)$, or equivalently, by \eqref{Pisom} and \eqref{Pomcoh}, the rank of the $\om$ map on $H^k_{dR}(M)$ as expressed in the second line of \eqref{rkdef}.

We recall the standard Morse inequalities bound for the Betti numbers in terms of the critical points of a Morse function $f$:
\begin{align}\label{cmi1}
(weak~Morse~inequalities)\:~\qquad\qquad\qquad~  b_k&\leq m_k\,,\qquad\qquad~~\: k=0,1,\ldots, 2n\,,\\
(strong~Morse~inequalities) \qquad~\sum_{i=0}^k (-1)^i b_k &\leq \sum_{i=0}^k (-1)^i m_k\,,\quad k=0,1,\ldots, 2n\,,\label{cmi11}
\end{align} 
where $m_k$ is the number of index $k$ critical points of $f$.  But since the $b_k^\om$'s can vary with the symplectic structure, we should expect that any sharp inequality bound of $b_k^\om$ should have dependence on the symplectic structure as well.  Hence, we introduce the rank of the map $\com$ on $C^k(M,f)$.
\begin{align}
v_k=\text{rank}\left(\com:C^k(M,f) \to C^{k+2}(M,f)\right)\,.
\end{align}
Notice that $v_k$ differs from $r_k$ with $r_k$ being the rank of the $\com$ map acting on the cohomology, $H^{k}_{C(f)}(M)$, and not on the cochain, $C^k(M,f)$.  However, since both the cochain and cohomology are generated by critical points, it is evident that 
\begin{align}\label{cmi2}
r_k &\leq v_k\,, \qquad\quad~ k=0, 1, \ldots, 2n-2\,,\\
r_k&= v_k=0\,, ~\,\quad k=2n-1, 2n\,.\label{cmi22}
\end{align}
Below is another key property: 
\begin{prop}\label{brmv}
Let $(M^{2n}, \om, f, g)$ be a closed symplectic manifold with the Morse function and the Riemannian metric satisfying the Morse-Smale transversality condition.  Then, we have the following inequality: 
\begin{align}\label{propeq}
b_k - r_{k-1} \leq m_k - v_{k-1}\,,\qquad k=0,1, \ldots, 2n+1\,.
\end{align} 
\end{prop}
\begin{proof}
For $k=2n+1$, each term in \eqref{propeq} vanishes and the inequality is trivial. For $k=2n$, with $r_{2n-1}=v_{2n-1}=0$ as noted in \eqref{cmi22}, the inequality is just the weak Morse inequality $b_{2n}\leq m_{2n}$.  In the remainder of the proof, we will only need to concern with the case $k=0, \dots, 2n-1$.  

Let us note that $C^k(M,f)$ is a cochain vector space over $\mathbb{R}$ and is finitely-generated by the critical points of index $k$.  Equipped with the differential $\del$, we can decompose $C^k(M,f)$ as follows: 
\begin{align*}
C^k = \del A^{k-1} \oplus B^k \oplus A^k\,,
\end{align*}
where $A^k= \frac{C^k}{\ker \del \cap C^k}$ is the space of cochains modulo cocycles, $\del A^{k-1}$ is the coboundary space, and $B^k$ is the space of cocycles modulo coboundaries (i.e. the cohomology). In this notation, the Betti number, $b_k =\dim B^k$.  And if we let $a_k = \dim A^k$, then $\dim \del A^{k-1}=a_{k-1}$, since $\del$ is an injective map on $A^{k-1}$.  Hence, 
\begin{align}\label{mkdecomp}
m_k = \dim C^k = \dim \del A^{k-1} +\dim B^k + \dim A^k =a_{k-1}+b_k+a_k\,.  
\end{align}
We consider the map $\com: C^{k-1} \to C^{k+1}$.  Expressing the action of $\com$ on each component of $C^{k-1} = \del A^{k-2} \oplus B^{k-1} \oplus A^{k-1}$, we observe that
\begin{align*}
&\text{(i)} ~\com \, \del A^{k-2} \,\subseteq\, \del A^k\\
&\text{(ii)} ~\com \, B^{k-1} \,\, \subseteq\, \del A^k \oplus B^{k+1}\\
&\text{(iii)}~ \com \, A^{k-1} \subseteq\, \del A^k \oplus B^{k+1} \oplus A^{k+1} = C^{k+1}
\end{align*}
having noted the property that $\com$ commutes with the differential $\del$ as in \eqref{MLeibniz}, and hence, $\com$ maps coboundaries to coboundaries (i), and also cocycles to cocycles (ii).  This can be expressed as a matrix operator 
\begin{align*}
\com 
\begin{bmatrix} \del A^{k-2} \\ B^{k-1} \\ A^{k-1} \end{bmatrix}
=
\begin{bmatrix}
R_{11}& R_{12} & R_{13} \\
O & R_{22} & R_{23} \\
O & O & R_{33}   \\
\end{bmatrix}
\begin{bmatrix} \del A^{k-2} \\ B^{k-1} \\ A^{k-1} \end{bmatrix}
\subseteq
\begin{bmatrix} \del A^k \\ B^{k+1} \\ A^{k+1} \end{bmatrix}
\end{align*}
Consider now $v_{k-1}$ which is the rank of the above matrix.  Note first that $r_{k-1} = \text{rank} ([\com]: B^{k-1} \to B^{k+1})$, and therefore, $r_{k-1}$ is the rank of $R_{22}$. Now the upper left block submatrix 
\[ \begin{bmatrix} R_{11} & R_{12} \\ O & R_{22}\end{bmatrix} \]
must have a rank that is greater than or equal to the rank of $R_{22}$.  So let  
\[\text{rank}\left( \begin{bmatrix} R_{11} & R_{12} \\ O & R_{22}\end{bmatrix}\right)=r_{k-1}+u_{k-1}\,,\quad \text{with~} u_{k-1} \geq 0\,. \] 
In particular, since there is a zero matrix in the lower left corner, only $R_{11}$ and $R_{12}$ can make $u_{k-1}$ greater than zero. However, both $R_{11}, R_{12}$ are in the first block-row, and thus can not contribute a rank greater than the size of the block-row, which is the dimension of $\del A^k$. Thus, we have the bound $0 \leq u_{k-1} \leq a_k$. Now $v_{k-1}$ is the rank of the whole matrix, which can not be less than $r_{k-1}+u_{k-1}$, so let 
\begin{align}\label{vkmo}
v_{k-1}=\text{rank}\left( \begin{bmatrix} R_{11} & R_{12} & R_{13} \\ O & R_{22} & R_{23} \\ O & O & R_{33}\end{bmatrix}\right)=r_{k-1}+u_{k-1}+t_{k-1}, \text{ with $t_{k-1} \geq 0$}
\end{align}
Again, the zero matrices in the third row mean that only $R_{13}, R_{23}$ and $R_{33}$ can make $t_{k-1}$ greater than zero. However, since these are in the third block-column, they cannot contribute a rank greater than the size of the block-column, which is the dimension of $A^{k-1}$, and we obtain the bound $0 \leq t_{k-1} \leq a_{k-1}$. Finally, combining \eqref{mkdecomp} and \eqref{vkmo}, we obtain 
\begin{align*}
m_k-v_{k-1}
&=(a_{k-1}+b_k+a_k) - (r_{k-1}+u_{k-1}+t_{k-1})\\
&=(b_k-r_{k-1})+(a_{k}-u_{k-1})+(a_{k-1}-t_{k-1})\\
&\geq b_k-r_{k-1}
\end{align*}
since the last two terms of the second line are both nonnegative. 
\end{proof}
Proposition \ref{brmv} above leads us to the desired cone Morse inequalities.
\begin{theorem}
Let $(M, \om, f, g)$ be a closed, symplectic manifold with the Morse function $f$ and the Riemannian metric $g$ satisfying the Morse-Smale transversality condition.  Then, we have the following inequalities for the dimensions $b_k^\om=\dim PH^k(M,\om)=\dim H^k(\tC)$: \\
(A) Weak cone Morse inequalities:
\begin{align}\label{wcMI3}
b_k^\om \;\leq \;m_k -v_{k-2} + m_{k-1} - v_{k-1}\,, \qquad k=0, 1, \ldots, 2n+1\,;
\end{align}
(B) Strong cone Morse inequalities:
\begin{align}\label{scMI3}
\sum_{i=0}^k\ (-1)^{k-i}\,b^\om_i \;
\leq \; m_k - v_{k-1}\,, \qquad\quad\: k=0, 1, \ldots, 2n+1\, .
\end{align}
\end{theorem}
\begin{proof}
We first prove the strong inequality which follows directly from the isomorphism, $H^k(\tC(\om))\cong H^k(\tC(\com))$, and Proposition \ref{brmv}. Specifically, the isomorphism implies the expression in \eqref{bok0}
\begin{align*}
b_i^\om = m_i - r_{i-2} - m_{i-1} - r_{i-1}\,.
\end{align*}
This gives a telescoping sum
\begin{align}
\sum_{i=0}^k\ (-1)^{k-i}\,b^\om_i  &=\sum_{i=0}^k \,(-1)^{k-i}\left(b_i - r_{i-2} + b_{i-1}-r_{i-1}\right)\nonumber\\
&= b_k - r_{k-1} \leq \; m_k - v_{k-1}\label{scMx} 
\end{align}
having applied the inequality of \eqref{propeq} which results in the desired strong cone Morse inequality.  

As for the weak inequality, it can be derived directly from the strong cone Morse equalities, or equivalently, from \eqref{propeq} with \eqref{bok0}.  Let us write out the $k$-th and $(k-1)$-th inequalities of \eqref{propeq},  
\[b_k - r_{k-1} \leq \; m_k - v_{k-1}\,,\qquad  \text{ and }\qquad b_{k-1} - r_{k-2} \leq \; m_{k-1} - v_{k-2}\,.\] 
Adding these two inequalities together and using \eqref{bok0} gives the desired weak cone Morse inequality
\begin{align}\label{wcMx}  
b_k^\om = b_k - r_{k-2} + b_{k-1}-r_{k-1} \leq m_k -v_{k-2} + m_{k-1} - v_{k-1}\,.
\end{align}
\end{proof}

Finally, let us consider the case when $f$ is a perfect Morse function.  By definition, a perfect Morse function implies $b_k=m_k$ for all values of $k$.  This means that $\dim H^k_{C(f)} = \dim C^k(M,f)$, and in particular, the Morse differential $\partial$ acts by zero on all generators of $C^*(M,f)$.  Clearly then,  when $f$ is perfect, we have both $b_k=m_k$ and also $r_k=v_k$.  We can therefore conclude that the weak cone Morse inequalities as expressed in  \eqref{wcMx} and the strong Morse inequalities as in \eqref{scMx} would both become equalities when $f$ is a perfect Morse function.  

Altogether, the weak and strong cone Morse inequalities and that they become equalities when $f$ is perfect are the statements of Theorem \ref{CMineq}. We have thus completed the proof of Theorem \ref{CMineq}.

\section{Examples}
In this section, we will consider on certain symplectic manifolds the cone Morse complex and check the cone Morse inequalities derived in the previous section:
\begin{align}\label{wcMI}
(weak)\quad\qquad\qquad \qquad\qquad 
b_k^{\om} \, &\leq \,
m_k-v_{k-2}+m_{k-1}-v_{k-1}\,, 
\\
\label{scMI}
(strong)\qquad \qquad 
\sum_{i=0}^k (-1)^{k-i}b^\om_i \,
&\leq 
\,m_k-v_{k-1}\,,
\end{align}    
for $k=0, 1, \ldots, 2n+1\,$.  In the first examples that we shall consider, the symplectic manifolds are K\"ahler. Due to the hard Lefschetz property, the wedge product map $[\om^j]:H^{n-j}_{dR}(M) \to H^{n+j}_{dR}(M)$ for $j=1, 2,  \ldots, n$, is an isomorphism.  
This implies in particular for $\txw_k = \text{rank } [\om] |_{H^k_{dR}(M)}\,$, that $\txw_k= \min(b_k, b_{k+2})$. It thus follows from the relation \cite{TTY, TT}
\begin{align}\label{bpsik}
H^k(\tC(\om)) \cong \coker\left[\om:H_{dR}^{k-2}(M) \to H_{dR}^{k}(M)\right] \oplus \ker\left[\om: H_{dR}^{k-1}(M) \to H_{dR}^{k+1}(M)\right]\,
\end{align}  that
\begin{align}\label{Kbom}
b^{\om}_k=\dim H^k(\tC(\om))=
   \begin{cases}
b_k-b_{k-2} & 0 \leq k \leq n\,, \\ 
b_{k-1}-b_{k+1}\ &  n +1 \leq k \leq {2n+1}\,,\\
\end{cases}
\end{align}
which are determined solely by the Betti numbers and do not vary with the symplectic structure.  This is special to K\"ahler symplectic manifolds, as generally, the $b^\om_k$'s can vary with the class $[\om]\in H^2_{dR}(M)$ (for explicit examples, see \cite{TY2, TW}).

\begin{rmk}In the special case where the Morse function $f$ and Riemannian metric $g$ are chosen such that $\com^k:C^{n-k}(M,f) \to C^{n+k}(M,f)$ is bijective, mirroring the Lefschetz property but on the level of the Morse cochains, then $v_k = \text{rank } c(\omega)=\min(m_k, m_{k+2})$. The weak cone Morse inequalities would then also be analogous to \eqref{Kbom}
\begin{align}\label{KmIneq}
b_k^\om \leq \begin{cases}
m_k-m_{k-2} & 0 \leq k \leq n\,, \\ 
m_{k-1}-m_{k+1} & n +1 \leq k \leq {2n+1}\,.
\end{cases}
\end{align}
While this does not hold generally, it always occurs when using a perfect Morse function manifolds on K\"ahler manifolds, which is the setting of our two K\"ahler examples below.
\end{rmk}

\begin{ex}
Consider $(\mathbb{CP}^n, \omega_{FS}=\frac{i} {2} \del \overline{\del}\log |z_i|^2)$, the complex $n$-dimensional projective space equipped with the Fubini-Study metric as the K\"{a}hler structure.  It is well-known to have a perfect Morse function that can be expressed as  
\[f([z_0, \ldots, z_{n}])=\frac{\sum \lambda_i |z_i|^2} {\sum |z_i|^2}\] 
such that $\lambda_i \neq \lambda_j$ for $i \neq j$. If we consider $\lambda_0<\lambda_1<\ldots<\lambda_n$, then the critical points are $p_{2j}=[0...:1:...0]$ with $1$ only in the $j$-th position and index $n_f(p_{2j})=2j$. Because the index of all the critical points are even, the Morse differential $\partial: C^k(M, f) \to C^{k+1}(M, f)$ necessarily vanishes for all $k=0,\ldots,2n$. Thus, $f$ is no doubt a perfect Morse function and  
\begin{align}\label{cpnm}
m_k = b_k(\mathbb{CP}^n) =
\begin{cases}
~1 \qquad & 0\leq k\leq 2n\,,~ k~even\,,\\
~0 & otherwise\,.
\end{cases}
\end{align}
Regarding the cone Morse differential $\partial_C=\begin{pmatrix} \partial & c(\om_{FS}) \\ 0 & - \partial \end{pmatrix}$, it has a non-zero component coming from the $c(\om_{FS})$ map.  We note that  
$$\CM(p_{2j}, p_{2j+2})=\{[0:\ldots:z_j:z_{j+1}:\ldots:0]: (z_j, 
z_{j+1}) \in \mathbb{C}^2 \setminus \{0\}\}$$ 
is isomorphic to $\mathbb{CP}^1$. 
Therefore, 

\begin{align*}
c(\omega_{FS})p_{2j}=\left(\int_{\mathbb{CP}^1} \omega_{FS}\right)p_{2j+2}=  \pi\, p_{2j+2}\,,
\end{align*}
and thus,
\begin{align}\label{vval}
v_{2j}=\rk c(\omega_{FS})|_{C^{2j}(
f)}=1\,, \qquad j=0, \ldots, n-1\,,  
\end{align} 
which is the same as $\txw_{2j}=\rk \om |_{H^{2j}(\mathbb{CP}^n)}$.
The cone Morse complex and cohomology can then be easily computed and we find 
\begin{align}\label{cpns}
b^\om_k= 
\begin{cases} 
~1  \qquad&  k=0, 2n+1\,,   \\
~0\qquad&  otherwise\,,
\end{cases} 
\end{align}
which agrees exactly with the expectation from \eqref{Kbom}.

The cone Morse inequalities \eqref{wcMI}-\eqref{scMI} can similarly be straightforwardly checked using \eqref{cpnm}-\eqref{vval}, and they are, in fact, equalities, as would be expected for a perfect Morse function.  
\end{ex}

\begin{ex}
Consider ($T^4=\mathbb{R}^4/\mathbb{Z}^4$, $\om=dx_1\w dx_2 + dx_3 \w dx_4$), the four-torus described using Euclidean coordinates, $x_i$ with identification $x_i \sim x_i + 1$ for $i=1,2,3,4$.  For this example, we will compute the $\tC(c(\om))$ complex with respect to the flat metric, $g=\sum dx_i^2$, and the Morse function is taken to be 
\begin{align}\label{t4f}
f=2-\frac{1}{2}\sum_{i=1}^4\cos(2\pi x_i)\,.
\end{align}
This Morse function has several desirable properties that are straightforward to prove:
\begin{itemize}
\item[(i)] the non-degenerate critical points are located at $x_i=[0]$ or $x_i=[\tf]$ and have Morse index equal to the number of coordinates which are equal to $[\tf]$;
\item[(ii)] the number of critical points of index $k$, $m_k=b_k(T^4)$ for all $k$.  Hence, $f$ is perfect and the  Morse differential $\del$ acts by zero; 
\item[(iii)] the pair $(f,g)$ satisfies Smale transversality. 
\end{itemize}
Because of (ii), the $\del_C$ map 
reduces to the $c(\om)$ map.  Hence, we are interested in pairs of critical points whose indices differ by two, e.g. $q_{k+1}$ has two more $[\tf]$ coordinates than $q_{k-1}$. Also, note that $\CM(q_{k+1},q_{k-1})$ will be a two-dimensional face with two of the coordinates fixed and two coordinates spanning the entire coordinate interval $[0,1]$ when we take the closure. 
\begin{table}[h]
    \centering
    $\left.\mspace{-43mu} \right.$\begin{tabular}{|c|c|c|c|}
     \hline   $k$ & $0$ & $1$ & $2$ \\
     \hline   $H^k(\tC(\omega))$ & $\begin{pmatrix} 1 \\ 0\end{pmatrix}$ &  $\begin{pmatrix}dx_1 \\ 0\end{pmatrix}, \begin{pmatrix}dx_2 \\ 0\end{pmatrix}$ & $\begin{pmatrix}dx_{13} \\ 0\end{pmatrix}, \begin{pmatrix}dx_{14} \\ 0\end{pmatrix}, \begin{pmatrix}dx_{23} \\ 0\end{pmatrix}, $ \\
       & & $\begin{pmatrix}dx_3 \\ 0\end{pmatrix}, \begin{pmatrix}dx_4 \\ 0\end{pmatrix}$ & $\begin{pmatrix}dx_{24} \\ 0\end{pmatrix}, \begin{pmatrix}dx_{12}-dx_{34} \\ 0\end{pmatrix}$ \\ 
       \hline   $H^k(\tC(\com))$ & $\begin{pmatrix} q_0 \\ 0\end{pmatrix}$ &  $\begin{pmatrix}q_1 \\ 0\end{pmatrix}, \begin{pmatrix}q_2 \\ 0\end{pmatrix}$ & $\begin{pmatrix}q_{13} \\ 0\end{pmatrix}, \begin{pmatrix}q_{14} \\ 0\end{pmatrix}, \begin{pmatrix}q_{23} \\ 0\end{pmatrix}, $ \\
       & & $\begin{pmatrix}q_3 \\ 0\end{pmatrix}, \begin{pmatrix}q_4 \\ 0\end{pmatrix}$ & $\begin{pmatrix}q_{24} \\ 0\end{pmatrix}, \begin{pmatrix}q_{12}-q_{34} \\ 0\end{pmatrix}$ \\
       \hline
    \end{tabular}
    $\left.\mspace{38mu} \right.$\begin{tabular}{|c|c|c|c|}
     \hline   $k$ & $3$ & $4$ & $5$ \\
     \hline   $H^k(\tC(\omega))$ & $ \begin{pmatrix}dx_{123} \\ 0\end{pmatrix}, \begin{pmatrix}dx_{124} \\ 0\end{pmatrix}, \begin{pmatrix}dx_{234} \\ 0\end{pmatrix},$ & $\begin{pmatrix}0 \\ dx_{123}\end{pmatrix}, \begin{pmatrix}0 \\ dx_{124} \end{pmatrix}$ & $\begin{pmatrix}0 \\ dx_{1234}\end{pmatrix}$   \\
       & $\begin{pmatrix}dx_{234} \\ 0\end{pmatrix}, \begin{pmatrix}0 \\ dx_{12} - dx_{34}\end{pmatrix}$ & $\begin{pmatrix}0 \\ dx_{134} \end{pmatrix}, \begin{pmatrix}0 \\ dx_{234}\end{pmatrix}$ &  \\
       \hline   $H^k(\tC(\com))$ & $ \begin{pmatrix}q_{123} \\ 0\end{pmatrix}, \begin{pmatrix}q_{124} \\ 0\end{pmatrix}, \begin{pmatrix}q_{234} \\ 0\end{pmatrix},$ & $\begin{pmatrix}0 \\ q_{123}\end{pmatrix}, \begin{pmatrix}0 \\ q_{124} \end{pmatrix}$ & $\begin{pmatrix}0 \\ q_{1234}\end{pmatrix}$   \\
       & $\begin{pmatrix}q_{234} \\ 0\end{pmatrix}, \begin{pmatrix}0 \\ q_{12} - q_{34}\end{pmatrix}$ & $\begin{pmatrix}0 \\ q_{134} \end{pmatrix}, \begin{pmatrix}0 \\ q_{234}\end{pmatrix}$ &  \\
       \hline
    \end{tabular}
    \caption{Cohomology of Cone$(\omega)$ versus Cone$(\com)$ on $(T^4,\om = dx_1 \w dx_2 +dx_3 \w dx_4)$.}
    \label{CohTab}
\end{table}
In Table \ref{CohTab}, we give the cohomologies of $H(\tC(c(\omega)))$ and $H(\tC(\omega))$. 
We use a multi-index notation of $I=\{i_1...i_j\}$ in increasing order such that $dx_{I}=dx_{i_1}\wedge ... \wedge dx_{i_j}$, $q_0$ denotes the index 0 point, and $q_{I}$ denotes the point with $\tf$ in entry $i_1, ... i_j$, i.e. $q_{13}=q_{\left[\tf,0, \tf, 0\right]}$.  The orientation of the submanifolds are chosen such that $\mP dx_{I}=q_{I}$. (c.f. Definition \ref{mP}.)  

Notice that $\com q_{I}$ only picks out critical points  that have two coordinates of $q_{I}$ changed from $[0]$ to $[\tf]$ in either the 1-2  or  3-4 directions. Thus, we find that 
\begin{align*}
\com q_{0}&=q_{12}+q_{34}\,, & \com q_{12}&=q_{1234}\,,&
\com q_{34}&=q_{1234}\,,& &
\\
\com q_{1}&=q_{134}\,,& 
\com q_{2}&=q_{234}\,, &
\com q_{3}&=q_{123}\,,&
\com q_{4}&=q_{124}\,,
 \end{align*} 
with all other critical points mapped to zero when acted upon by $\com$. 

It is straightforward to see from above that $v_k = r_k$ and  that cone Morse inequalities give the equalities $b^\om_k = m_k - v_{k-2}-m_{k-1}-v_{k-1}\,$ for $0\leq k\leq 5\,$.  This is as expected with $f$ in \eqref{t4f} being a perfect Morse function.

\end{ex}

\medskip

Next, we consider a non-K\"ahler symplectic manifold where the hard Lefschetz property does not hold. 
\begin{ex}
Let $(M, \om)$ be the six-dimensional, closed, symplectic manifold constructed by Cho in \cite{Cho} where the symplectic form $\om$ is not hard Lefschetz type.   Topologically, $M$ can be described as a two-sphere bundle over a projective $K3$ surface and also has the following properties \cite[Theorem 1.3]{Cho}: (i) $M$ is simply-connected; (ii) the odd degree cohomologies vanish, i.e. $H_{dR}^{1}(M)=H_{dR}^3(M)=H_{dR}^5(M)=0\,$. 

Consider the cohomology $PH(M, \om)\cong H(\tC(\om))$.  From \eqref{bpsik}, we find
\begin{align*}
b^\om_0&= b^\om_7=1\,,\\
b^\om_1&=b^\om_6=0\,,\\
b^\om_2&=b^\om_5=b_2(X)-1\,,\\
b^\om_3&=\dim\left[\ker\left(\om: H^2(X)\to H^4(X)\right)\right]>0\,,\\
b^\om_4&=\dim\left[{\rm coker}\left(\om: H^2(X) \to H^4(X)\right)\right]>0\,.
\end{align*}
Note that $b^\om_3=b^\om_4 > 0$ since $(M, \om)$ is not hard Lefschetz, which implies that the map, $\om: H_{dR}^2(M) \to H_{dR}^4(M)$, can not be an isomorphism.

For the cone Morse complex and inequalities, we can again choose to work with a perfect Morse function on $M$.  That such exists is due to a a result of Smale \cite[Theorem 6.3]{Smale} which states that any simply-connected manifold of dimension greater than five that has no homology torsion has a perfect Morse function.  (No homology torsion here can be seen from applying the Gysin sequence to $M$ as a two-sphere bundle over $K3$.)  Since $M$ has trivial odd-degree cohomology, this implies that 
\begin{align*}
m_0&=m_6=1\,,\\
m_1&=m_3=m_5=0\,,\\
m_2&=m_4=b_2(M)\,.
\end{align*}
It is straightforward to check that the bounds \eqref{wcMI}-\eqref{scMI} are satisfied.  In particular, for the weak cone Morse bound of $\eqref{wcMI}$, the $k=3,4$ case corresponds to
\begin{align*}
b^\om_3&\leq m_3 + m_2 -v_2 = m_2 - v_2\,,\\
b^\om_4&\leq m_4 + m_3 -v_2 = m_4 - v_2\,,
\end{align*}
The above demonstrates the necessity of having both the $m_k$ and the $m_{k-1}$ term in the symplectic cone Morse inequalities.
\end{ex}

\begin{rmk}\label{Sbound}
We comment that there is a preprint \cite{Machon} that presents some symplectic Morse-type inequalities which are different from those here and actually not valid generally.  For instance, the inequality in \cite[Corollary 3]{Machon} can be expressed in our notation as $\dim F^pH^{n+p+1}(M,\om) \leq m_{n-p}$, which is not satisfied in the above Cho's non-K\"ahler six-dimensional example $(M,\om)$ for a perfect Morse function.  Specifically, it gives for $p=0$ case the inequality relation, $b^\om_4\leq m_3 =0\,,$ which is inconsistent with $b^\om_4>0$ with $\om$ being of non-hard Lefschetz type.
\end{rmk}

\section{Discussion}

Thus far, in this paper, we have for simplicity focused on the $p=0$ case of the TTY cohomologies, $F^pH(M, \om) = H(\tC(\om^{p+1}))$.  
Let us comment in this final section the $p>0$ case and lay out the results which generalize the $p=0$ case.  The cone Morse theory in the $p>0$ case can be considered analytically similar to the computations in this paper though the calculations are more involved.  In general, the TTY cohomologies for $p=0, 1, \ldots, n-1$ algebraically correspond to   \cite{TTY,TT}:
\begin{align}\label{kerco}
F^pH^k(M, \om)\cong 
H^k(\tC(\omega^{p+1}))&\cong 
\coker\left([\omega^{p+1}]:H_{dR}^{k-2p-2}(M) \to H_{dR}^{k}(M) \right)
 \nonumber \\
& \quad\quad  \oplus \,\ker\left([\omega^{p+1}]:H_{dR}^{k-2p-1}(M) \to H_{dR}^{k+1}(M) \right) 
\end{align}
with $k=0, 1, \ldots, 2n+2p+1\,$.
The relevant cone complex for the general $p$ case would have the elements and differential 
\begin{align}\label{Cpcmplx}
    \tC^k(\om^{p+1}) = \Om^k(M) \oplus \theta \Om^{k-2p-1}(M)\,,\qquad d_C = \begin{pmatrix} d & \om^{p+1}\\ 0 & -d\end{pmatrix},
\end{align}
where $\theta$ is now a formal $(2p+1)$-form such that $d\theta = \om^{p+1}$.
And the corresponding cone Morse cochain complex would be
\begin{align}\label{cCpcmplx}
\tC^k(c(\om^{p+1}))=C^k(M,f)\oplus C^{k-2p-1}(M,f)\,,\qquad \del_C= \begin{pmatrix} \del & c(\om^{p+1}) \\ 0 & -\del \end{pmatrix}\,,
\end{align}
with $c(\om^{p+1}): C^k(M,f)\to C^{k+2p+2}(M,f)$ given by
\begin{align}
c(\om^{p+1})\,q_{k}=\sum_{r_{k+2p+2}}\left( \int_{\overline{\CM(r_{k+2p+2}, q_k)}}\, \om^{p+1} \right)r_{k+2p+2}\,,
\end{align}
which integrates $\om^{p+1}$ over the $2(p+1)$-dimensional submanifold $\overline{\CM(r_{k+2p+2}, q_k)}$ of gradient flow lines from the index $k+2p+2$ critical point, $r_{k+2p+2}\,$, to $q_k$.

We can take the inner product on $\tC^k(\om^{p+1})$,  just as in \eqref{cinprod}, to be
\begin{align}
\langle\eta_k + \theta \xi_{k-2p-1}, \eta'_k + \theta \xi'_{k-2p-1}\rangle_C \,= \, \langle\eta_k, \eta'_k\rangle + \langle\xi_{k-2p-1}, \xi'_{k-2p-1}\rangle\,. 
\end{align}
This inner product defines the adjoint operator $d^*_C$  
that goes into the cone Laplacian $\Delta_C = d_C d^*_C + d^*_C d_C\,$, which is a second-order elliptic operator on $\tC^k(\om^{p+1})$.  The Witten deformation method can be applied to this cone Laplacian $\Delta_C$ for $p>0$ following the steps described in Section 2 and 3.  The calculations are similar to the $p=0$ case.  At large $t$, there are again only two generators to the solutions of the deformed harmonic Laplacian localized at each critical point $p\in Crit(f)$.  For instance, for $k\leq n+p$, the generators take the form 
\begin{align*}
\begin{pmatrix} \zeta_k \\ 0 \end{pmatrix}, \qquad 
\begin{pmatrix} -\tau \w \om^p \w \zeta_{k-2p-1} \\ \zeta_{k-2p-1}\end{pmatrix},
\end{align*}
which generalizes the generators in \eqref{kln1} and \eqref{kln2}, respectively.  Generalizations of the estimates similar to those for the $p=0$ case can be carried out which results in the isomorphism of the cohomologies of the cone complex \eqref{Cpcmplx} with that of the cone Morse complex \eqref{cCpcmplx}.  From these results, we can likewise obtain cone Morse inequalities.  
Explicitly, using the notation  
\begin{align*}
s^p_k= \dim F^pH^k(M, \om)
=\dim H^k(\tC(\om^{p+1}))
\end{align*}
to denote the dimension of the TTY cohomology, we expect the following weak and strong cone Morse inequalities for all $p=0, 1, \dots, n-1$:
\begin{align}
s^p_k \, &\leq \,
m_k-v_{k-2p-2}+m_{k-2p-1}-v_{k-2p-1}\,,\\  
\sum_{i=0}^k (-1)^{k-i}s^p_i \, &\leq 
\left(\sum_{i=k-2p}^k  (-1)^{k-i}m_i\right) - \ v_{k-2p-1}\,,
\end{align}    
where $k= 0,1, \ldots, 2n+2p+1\,$ and 
\begin{align}
v_k=\rk\left(c(\om^{p+1}):C^k(M,f) \to C^{k+2p+2}(M,f)\right).
\end{align}


\vskip 1cm

\noindent
{Department of Mathematics, University of California, Riverside, CA 92507, USA}\\
{\it Email address:}~{\tt dclausen@ucr.edu}
\vskip .5 cm
\noindent
{Department of Mathematics, Washington University, St. Louis, MO 63130, USA}\\
{\it Email address:}~{\tt xtang@math.wustl.edu}
\vskip .5 cm
\noindent
{Department of Mathematics, University of California, Irvine, CA 92697, USA}\\
{\it Email address:}~{\tt lstseng@uci.edu}


\begin{bibdiv}
\begin{biblist}[\normalsize]

\bib{AB}{article}{
   author={Austin, D. M.},
   author={Braam, P. J.},
   title={Morse-Bott theory and equivariant cohomology},
   conference={
      title={The Floer Memorial Volume},
   },
   book={
      series={Progr. Math.},
      volume={133},
      publisher={Birkh\"{a}user, Basel},
   },
   date={1995},
   pages={123--183},
}

\bib{BL}{article}{
   author={Bismut, J.-M.},
   author={Lebeau, G.},
   title={Complex immersions and Quillen metrics},
   journal={Inst. Hautes \'{E}tudes Sci. Publ. Math.},
   number={74},
   date={1991},
   pages={1--297},
}

\bib{BZ}{article}{
   author={Bismut, J.-M.},
   author={Zhang, W.},
   title={An extension of a theorem by Cheeger and M\"{u}ller},
   note={With an appendix by Fran\c{c}ois Laudenbach},
   journal={Ast\'{e}risque},
   number={205},
   date={1992},
   pages={235},
   issn={0303-1179},
}

\bib{Cho}{article}{
   author={Cho, Y.},
   title={Hard Lefschetz property of symplectic structures on compact K\"{a}hler
   manifolds},
   journal={Trans. Amer. Math. Soc.},
   volume={368},
   date={2016},
   number={11},
   pages={8223--8248},
   issn={0002-9947},
}


\bib{CTT3}{article}{
    author={Clausen, D.},
    author={Tseng, L.-S.},
    author={Tang, X.},
    title={Mapping cone and Morse theory},
    note = {arXiv:2405.02272 [math.DG]},
}



\bib{GTV}{article}{
   author={Gibson, M.},
   author={Tseng, L.-S.},
   author={Vidussi, S.},
   title={Symplectic structures with non-isomorphic primitive cohomology on open 4-manifolds},
   journal={Trans. Amer. Math. Soc.},
   volume={375},
   date={2022},
   number={12},
   pages={8399--8422},
}

\bib{GS}{article}{
   author={Guillemin, V.},
   author={Sternberg, S.},
   title={Geometric quantization and multiplicities of group
   representations},
   journal={Invent. Math.},
   volume={67},
   date={1982},
   number={3},
   pages={515--538},
}

\bib{MaZ}{article}{
   author={Ma, X.},
   author={Zhang, W.},
   title={Geometric quantization for proper moment maps: the Vergne
   conjecture},
   journal={Acta Math.},
   volume={212},
   date={2014},
   number={1},
   pages={11--57},
}

\bib{Machon}{article}{
  author={Machon, T.},
  title = {Some Morse-type inequalities for symplectic manifolds},
   note={\!arXiv:2109.13010v1 [math.SG]},
}


\bib{Smale}{article}{
   author={Smale, S.},
   title={On the structure of manifolds},
   journal={Amer. J. Math.},
   volume={84},
   date={1962},
   pages={387--399},
   issn={0002-9327},
}

\bib{Stratmann}{article}{
   author={Stratmann, B.},
   title={Nowhere vanishing primitive of a symplectic form},
   journal={Asian J. Math.},
   volume={26},
   date={2022},
   number={5},
   pages={705--708},
   issn={1093-6106},
}



\bib{TT}{article}{
   author={Tanaka, H. L.},
   author={Tseng, L.-S.},
   title={Odd sphere bundles, symplectic manifolds, and their intersection
   theory},
   journal={Camb. J. Math.},
   volume={6},
   date={2018},
   number={3},
   pages={213--266},
   issn={2168-0930},
}

\bib{TiZ}{article}{
   author={Tian, Y.},
   author={Zhang, W.},
   title={An analytic proof of the geometric quantization conjecture of
   Guillemin-Sternberg},
   journal={Invent. Math.},
   volume={132},
   date={1998},
   number={2},
   pages={229--259},
}


\bib{TTY}{article}{
   author={Tsai, C.-J.},
   author={Tseng, L.-S.},
   author={Yau, S.-T.},
   title={Cohomology and Hodge theory on symplectic manifolds: III},
   journal={J. Differential Geom.},
   volume={103},
   date={2016},
   number={1},
   pages={83--143},
   issn={0022-040X},
}


\bib{TW}{article}{
   author={Tseng, L.-S.},
   author={Wang, L.},
   title={Symplectic boundary conditions and cohomology},
   journal={J. Differential Geom.},
   volume={122},
   date={2022},
   number={2},
   pages={271--340},
}

\bib{TY1}{article}{
   author={Tseng, L.-S.},
   author={Yau, S.-T.},
   title={Cohomology and Hodge theory on symplectic manifolds: I},
   journal={J. Differential Geom.},
   volume={91},
   date={2012},
   number={3},
   pages={383--416},
}

\bib{TY2}{article}{
   author={Tseng, L.-S.},
   author={Yau, S.-T.},
   title={Cohomology and Hodge theory on symplectic manifolds: II},
   journal={J. Differential Geom.},
   volume={91},
   date={2012},
   number={3},
   pages={417--443},
}

\bib{Viterbo}{article}{
   author={Viterbo, C.},
   title={The cup-product on the Thom-Smale-Witten complex, and Floer cohomology},
   conference={title={The Floer memorial volume},},
   book={series={Progr. Math.},
        volume={133}, 
        publisher={Birkh\"{a}user, Basel},},
   date={1995},
   pages={609--625},
}

\bib{Witten}{article}{
   author={Witten, E.},
   title={Supersymmetry and Morse theory},
   journal={J. Differential Geom.},
   volume={17},
   date={1982},
   number={4},
   pages={661--692 (1983)},
}


\bib{Zhang}{book}{
   author={Zhang, W.},
   title={Lectures on Chern-Weil theory and Witten deformations},
   series={Nankai Tracts in Mathematics},
   volume={4},
   publisher={World Scientific Publishing Co., Inc., River Edge, NJ},
   date={2001},
   pages={xii+117},
}
\bib{HaoZhuang}{article}{
  author={Zhuang, H.},
  title = {Invariant Morse-Bott-Smale cohomology and the Witten deformation},
  note={arXiv:2311.10417v2 [math.DG]},
}


\end{biblist}
\end{bibdiv}
\end{document}